\documentclass[final,leqno]{siamltex}

\usepackage{mathtools}
\usepackage{amssymb}
\usepackage{stmaryrd}
\usepackage{color}
\usepackage{pdflscape}
\usepackage{epsfig}

\usepackage{epstopdf}
\usepackage{multirow}
\usepackage{hhline}
\usepackage{txfonts}
\usepackage{mathrsfs}
\usepackage{color}

\usepackage[margin=0.9in]{geometry}
\usepackage{subcaption}

\usepackage{mathrsfs}
\newtheorem{thm}{Theorem}[section]

\newtheorem{lem}[thm]{Lemma}
\newtheorem{rem}[thm]{Remark}
\newtheorem{assum}{Assumption}

\numberwithin{equation}{section}
\title{Petrov-Galerkin Method for Fully Distributed-Order \\ Fractional Partial Differential Equations 
	\thanks{This work was supported by the AFOSR Young Investigator Program (YIP) award (FA9550-17-1-0150) and by the MURI/ARO on Fractional PDEs for Conservation Laws and Beyond: Theory, Numerics and Applications (W911NF- 15-1-0562) and by USA National Science Foundation grants DMS-1462156 and EAR-1344280.}}

\author{
	Mehdi Samiee
	\footnote{D\lowercase{epartment of} C\lowercase{omputational} M\lowercase{athematics}, S\lowercase{cience}, \lowercase{and}, E\lowercase{ngineering} \& D\lowercase{epartment of} M\lowercase{echanical} E\lowercase{ngineering},	
		M\lowercase{ichigan} S\lowercase{tate} U\lowercase{niversity}, 428 S S\lowercase{haw} L\lowercase{ane}, E\lowercase{ast} L\lowercase{ansing}, MI 48824, USA}
	,Ehsan Kharazmi
	\footnote{D\lowercase{epartment of} C\lowercase{omputational} M\lowercase{athematics}, S\lowercase{cience}, \lowercase{and}, E\lowercase{ngineering} \& D\lowercase{epartment of} M\lowercase{echanical} E\lowercase{ngineering},	
		M\lowercase{ichigan} S\lowercase{tate} U\lowercase{niversity}, 428 S S\lowercase{haw} L\lowercase{ane}, E\lowercase{ast} L\lowercase{ansing}, MI 48824, USA}
	, Mohsen Zayernouri
	\footnote{D\lowercase{epartment of} C\lowercase{omputational} M\lowercase{athematics}, S\lowercase{cience}, \lowercase{and}, E\lowercase{ngineering} \&
		D\lowercase{epartment of} M\lowercase{echanical} E\lowercase{ngineering},	
		M\lowercase{ichigan} S\lowercase{tate} U\lowercase{niversity}, 428 S S\lowercase{haw} L\lowercase{ane}, E\lowercase{ast} L\lowercase{ansing}, MI 48824, USA,  C\lowercase{orresponding author; zayern@msu.edu}}
	,and Mark M Meerschaert
	\footnote{D\lowercase{epartment of} S\lowercase{tstistics} and P\lowercase{robability},	
		M\lowercase{ichigan} S\lowercase{tate} U\lowercase{niversity}, 619 R\lowercase{ed} C\lowercase{edar} R\lowercase{oad}, E\lowercase{ast} L\lowercase{ansing}, MI 48824, USA}
}

\begin{document}

\maketitle

\begin{abstract}
	Distributed-order PDEs are tractable mathematical models for complex multiscaling anomalous transport, where derivative orders are distributed over a range of values. We develop a fast and stable Petrov-Galerkin spectral method for such models by employing Jacobi \textit{poly-fractonomial}s and Legendre polynomials as temporal and spatial basis/test functions, respectively. By defining the proper underlying \textit{ distributed Sobolev} spaces and their equivalent norms, we prove the well-posedness of the weak formulation, and thereby carry out the corresponding stability and error analysis. We finally provide several numerical simulations to study the performance and convergence of proposed scheme.
\end{abstract}

\begin{keywords}
	\textit{Distributed Sobolev} space, well-posedness analysis, discrete \textit{inf-sup} condition, spectral convergence, Jacobi \textit{poly-fractonomials}, Legendre polynomials
\end{keywords}
\pagestyle{myheadings}
\thispagestyle{plain}

\section{Introduction}
\label{Sec: Intro}
%%%%%%%%%%%%%%%%%%%%%%%%%%%%%%%%%%%%
%

Over the past decades, anomalous transport has been observed and investigated in a wide range of applications such as turbulence \cite{shraiman2000scalar,metzler2014anomalous,iwayama2015anomalous,cheng2015eulerian}, porous media \cite{tyukhova2016mechanisms,ardakani2016investigation,zhang2016backward,edery2015anomalous,zhang2015modeling,Benson2000}, geoscience \cite{armour2016southern}, bioscience \cite{naghibolhosseini2015estimation,naghibolhosseini2017fractional,perdikaris2014fractional,  regner2014randomness}, and viscoelastic material \cite{suzuki2016fractional,goychuk2015anomalous,mashelkar1980anomalous}. The underlying anomalous features, manifesting in memory-effects, non-local interactions, power-law distributions, sharp peaks, and self-similar structures, can be well-described by fractional partial differential equations (FPDEs) \cite{meerschaert2012fractional, meerschaert2012stochastic,Klages2008, metzler2000random}. % move some of the related references here.
However, in cases where a single power-law scaling is not observed over the whole domain, the processes cannot be characterized by a fixed fractional order \cite{sokolov2004distributed}. Examples include accelerating superdiffusion, decelerating subdiffusion \cite{gorenflo2015time,sokolov2004distributed}, and random processes subordinated to Wiener processes \cite{duan2017steady,konjik2017distributed,meerschaert2012stochastic,eab2011fractional,mainardi2008time,mainardi2007two,chechkin2002retarding}. A faithful description of such anomalous transport requires exploiting distributed-order derivatives, in which the derivative order has a distribution over a range of values. 

%We generalize the existing FPDEs in \cite{samiee2017Unified,kharazmi2016petrov} by letting the derivative orders take a distribution over a range of values.

Numerical methods for FPDEs, which can exhibit history dependence and non-local features have been recently addressed by developing finite-element methods \cite{jin2017nonnegativity,ainsworth2017aspects}, spectral/spectral-element methods \cite{yamamoto2018weak,chen2017laguerre,mao2016efficient,samiee2017Unified,mao2017spectral,kharazmi2017petrov}, and also finite-difference and finite-volume methods \cite{coronel2018numerical,macias2018explicit,ammi2018finite}. Distributed-order FPDEs impose further complications in numerical analysis by introducing distribution functions, which require compliant underlying function spaces, as well as efficient and accurate integration techniques over the order of the fractional derivatives. In \cite{zaky2018legendre,Li2018Two,fan2018numerical,tian2015class,luchko2009boundary,jin2016error}, numerical analysis of distributed-order FPDEs was extensively investigated. More recently, Liao et al. \cite{liao2017stability} studied simulation of a distributed subdiffusion equation, approximating the distributed order Caputo derivative using piecewise-linear and quadratic interpolating polynomials. Abbaszadeh and Dehghan \cite{abbaszadeh2017improved} employed an alternating direction implicit approach, combined with an interpolating element-free Galerkin method, on distributed-order time-fractional diffusion-wave equations. Kharazmi and Zayernouri  \cite{kharazmi2017pseudo-spectral} developed a pseudo-spectral method of Petrov-Galerkin sense, employing nodal expansions in the weak formulation of distributed-order fractional PDEs. In \cite{kharazmi2016petrov}, they also introduced \textit{distributed Sobolev} space and developed two spectrally accurate schemes, namely, a Petrov--Galerkin spectral method and a 
spectral collocation method for distributed order fractional differential equations. Besides, Tomovski and Sandev \cite{tomovski2017distributed} investigated the solution of generalized distributed-order diffusion equations with fractional time-derivative, using the Fourier-Laplace transform method. 

The main purpose of this study is to develop and analyze a Petrov-Galerkin (PG) spectral method to solve a $(1+d)$-dimensional fully distributed-order FPDE with two-sided derivatives of the form 
\begin{align}
\label{Strong-form}
\nonumber
&
\int_{\tau^{min}}^{\tau^{max}} \varphi(\tau) \,\prescript{C}{0}{\mathcal{D}}_{t}^{2\tau} u \,\,  d\tau
+  
\sum_{i=1}^{d}  \int_{\mu_i^{min}}^{\mu_i^{max}} \varrho_i(\mu_i) \,
[c_{l_i}\prescript{RL}{a_i}{\mathcal{D}}_{x_i}^{2\mu_i} u^{} +c_{r_i}\prescript{RL}{x_i}{\mathcal{D}}_{b_i}^{2\mu_i} u^{} ] \,\,  d\mu_i
\\
& =  
\sum_{j=1}^{d} \int_{\nu_j^{min}}^{\nu_j^{max}} \rho_j(\nu_j) \,
[\kappa_{l_j}\prescript{RL}{a_j}{\mathcal{D}}_{x_j}^{2\nu_j} u^{} +\kappa_{r_j}\prescript{RL}{x_j}{\mathcal{D}}_{b_j}^{2\nu_j} u^{} ] \,\, d\nu_j
% \nonumber
-\gamma\,\, u^{}  +
f ,
\end{align}
%
%subject to homogeneous Dirichlet boundary conditions and zero initial condition, where $t\in [0,\, T]$, $x_j\in [a_j,\, b_j]$, $2\tau^{min}<2\tau^{max} \, \in (0, \, 2]$, $2\tau^{min}\neq 1, \,\,2\tau^{max} \neq 1$, $2\mu_i^{min}<2\mu_i^{max} \, \in (0, \, 1]$, and $2\nu_j^{min}<2\nu_j^{max} \, \in (1, \, 2]$ for $i,j=1, \, 2, \, ..., \, d$. The coefficients $ c_{l_i} $, $ c_{r_i} $, $ \kappa_{l_i} $, $ \kappa_{r_i} $, and $\gamma$ are considered to be constant. Besides, the distribution function $0<\varphi(\tau)\in L^1([\tau^{min},\tau^{max}])$ , $0<\varrho_i(\mu_i)\in L^1([\mu_i^{min},\mu_i^{max}])$, and $0<\rho_j(\nu_j)\in L^1([\nu_j^{min},\nu_j^{max}])$.
subject to homogeneous Dirichlet boundary conditions and zero initial condition, where for $i,j=1, \, 2, \, ..., \, d$
\vspace{-0.1 in}
\begin{align*}
& t\in [0,\, T], \quad  x_j\in [a_j,\, b_j], 
\\
& 2\tau^{min}<2\tau^{max} \, \in (0, \, 2], \,\, 2\tau^{min}\neq 1, \,\,2\tau^{max} \neq 1, 
\\
& 2\mu_i^{min}<2\mu_i^{max} \, \in (0, \, 1], \,\, 2\nu_j^{min}<2\nu_j^{max} \, \in (1, \, 2],
\\
& 0<\varphi(\tau)\in L^1((\tau^{min},\tau^{max})), \,\, 
0<\varrho_i(\mu_i)\in L^1((\mu_i^{min},\mu_i^{max})), \,\, 
0<\rho_j(\nu_j)\in L^1((\nu_j^{min},\nu_j^{max})),
\end{align*}
and the coefficients $ c_{l_i} $, $ c_{r_i} $, $ \kappa_{l_i} $, $ \kappa_{r_i} $, and $\gamma$ are constant. We briefly highlight the main contributions of this study as follows.
\begin{itemize}
	
	\item We consider fully distributed fractional PDEs as an extension of the existing fractional PDEs in \cite{samiee2017Unified,kharazmi2016petrov} by replacing the fractional operators by their corresponding distributed order ones. We further derive the weak formulation of the problem. 
	
	\item We construct the underlying function spaces by extending the \textit{distributed Sobolev} space in \cite{kharazmi2016petrov} to higher dimensions in time and space, endowed with equivalent associated norms. 
	
	\item We develop a Petrov-Galerkin spectral method, employing Legendre polynomials and Jacobi \textit{poly-fractonomials} \cite{zayernouri2013fractional} as spatial and temporal basis/test functions, respectively. We also formulate a fast solver for the corresponding weak form of \eqref{Strong-form}, following \cite{samiee2017Unified}, which significantly reduces the computational expenses in high-dimensional problems.
	
	\item We establish well-posedness of the weak form of the problem in the underlying \textit{distributed Sobolev} spaces  respecting the analysis in \cite{samiee2017unified1} and prove the stability of the proposed numerical scheme. We additionally perform the corresponding error analysis, where the  \textit{distributed Sobolev} spaces enable us to obtain accurate error estimate of the scheme. 	
\end{itemize}

\noindent We note that the model \eqref{Strong-form} includes distributed-order fractional diffusion and fractional advection-dispersion equations (FADEs) with constant coefficients on bounded domains, when the corresponding distributions $\varphi$, $\varrho_i$, and $\varrho_j$, $i,j=1,2,\cdots,d$ are chosen to be Dirac delta functions. To examine the performance and convergence of the developed PG method in solving different cases, we also perform several numerical simulations.

The paper is organized as follows: in Section 2, we introduce some preliminaries from fractional calculus. In Section 3, we present the mathematical framework of the bilinear form and carry out the corresponding well-posedness analysis. We construct the PG method for the discrete weak form problem and formulate the fast solver in Section 4. In Section 5, we perform the stability and error analysis in detail. In Section 6, we illustrate the convergence rate and the efficiency of method via numerical examples. We conclude the paper with a summary.

%%%%%%%%%%%%%%%%%%%%%%%%%%%%%%%%%%
%
\section{Preliminaries on Fractional Calculus}
\label{Sec: Notation}
%
%%%%%%%%%%%%%%%%%%%%%%%%%%%%%%%%%%
%
Recalling the definitions of the fractional derivatives and integrals from \cite{zayernouri2013fractional,meerschaert2012stochastic}, we denote by $\prescript{RL}{a}{\mathcal{D}}_{x}^{\sigma} g(x)$ and $\prescript{RL}{x}{\mathcal{D}}_{b}^{\sigma} g(x)$ the left-sided and the right-sided Reimann-Liouville fractional derivatives of order $\sigma>0$,
\begin{eqnarray}
\label{eqRL}
\prescript{RL}{a}{\mathcal{D}}_{x}^{\sigma} g(x) = \frac{1}{\Gamma(n-\sigma)}  \frac{d^{n}}{d x^n} \int_{a}^{x} \frac{g(s) }{(x - s)^{\sigma +1-n} }\,\,ds,\quad x \in [a,b],
\\
\label{eqRL2}
\prescript{RL}{x}{\mathcal{D}}_{b}^{\sigma} g(x) = \frac{(-1)^{n}}{\Gamma(n-\sigma)}   \frac{d^{n}}{d x^n} \int_{x}^{b} \frac{g(s) }{(s - x)^{\sigma +1-n} }\,\,ds,\quad x \in [a,b],
\end{eqnarray}
in which $g(x) \in L^1[a,b]$ and $\int_{a}^{x} \frac{g(s) }{(x - s)^{\sigma +1-n} }\,\,ds, \, \, \int_{a}^{x} \frac{g(s) }{(x - s)^{\sigma +1-n} }\,\,ds  \in C^{n}[a,b] $, respectively, where $n= \lceil \sigma \rceil$. Besides, $\prescript{C}{a}{\mathcal{D}}_{x}^{\sigma} g(x)$ and $\prescript{C}{x}{\mathcal{D}}_{b}^{\sigma} g(x)$ represent the left-sided and the right-sided Caputo fractional derivatives, where
\begin{eqnarray}
\label{eq6}
\prescript{C}{a}{\mathcal{D}}_{x}^{\sigma} f(x) \,=\, 
\frac{1}{\Gamma(n-\nu)} \int_{a}^{x} \frac{g^{(n)}(s) }{(x - s)^{\nu +1-n} }\,\,ds,\quad x \in [a,b], 
\\
\label{eq7}
\prescript{C}{x}{\mathcal{D}}_{b}^{\sigma} f(x) \,=\, 
\frac{(-1)^n}{\Gamma(n-\nu)} \int_{x}^{b} \frac{g^{(n)}(s) }{(s - x)^{\nu +1-n} }\,\,ds,\quad x \in [a,b].
\end{eqnarray}
The relationship between the RL and the Caputo fractional derivatives is given by
\begin{eqnarray}
\label{eq4}
\prescript{Rl}{a}{\mathcal{D}}_{x}^{\nu} f(x) \,=\, 
\frac{f ( a ) }{\Gamma(1-\nu)(x-a)^{\nu}}\,+\,
\prescript{C}{a}{\mathcal{D}}_{x}^{\nu} f(x) 
\\
\label{eq5}
\prescript{Rl}{x}{\mathcal{D}}_{b}^{\nu} f(x) \,=\, 
\frac{f ( b ) }{\Gamma(1-\nu)(b-x)^{\nu}}\,+\,
\prescript{C}{x}{\mathcal{D}}_{b}^{\nu} f(x),
\end{eqnarray}
when $\lceil \nu \rceil=1$, see e.g. (2.33) in \cite{meerschaert2012stochastic}. In the case of homogeneous boundary conditions, we obtain $\prescript{Rl}{a}{\mathcal{D}}_{x}^{\nu} f(x)=\prescript{C}{a}{\mathcal{D}}_{x}^{\nu} f(x):=\prescript{}{a}{\mathcal{D}}_{x}^{\nu} f(x)$ and $\prescript{Rl}{x}{\mathcal{D}}_{b}^{\nu} f(x)= \prescript{C}{x}{\mathcal{D}}_{b}^{\nu} f(x):=\prescript{}{x}{\mathcal{D}}_{b}^{\nu} f(x)$. The Reimann-Liouville fractional integrals of Jacobi \textit{poly-fractonomials} are analytically obtained in the standard domain $\xi \in [-1,1]$ as  \cite{zayernouri2013fractional,zayernouri2015unified}
\begin{eqnarray}
\label{6}
\prescript{RL}{-1}{\mathcal{I}}_{\xi}^{\sigma} \{ (1+\xi)^{\beta} P_{n}^{\alpha, \beta} {(\xi)}\} = \frac{\Gamma(n+\beta+1)}{\Gamma(n+\beta+\sigma+1)}\, (1+\xi)^{\beta+\sigma} P_{n}^{\alpha-\sigma,\beta+\sigma}{(\xi)},
\end{eqnarray}
and
\begin{eqnarray}
\label{7}
\prescript{RL}{\xi}{\mathcal{I}}_{1}^{\sigma} \{ (1-\xi)^{\alpha} P_{n}^{\alpha, \beta} {(\xi)}\} = \frac{\Gamma(n+\alpha+1)}{\Gamma(n+\alpha+\sigma+1)}\, (1-\xi)^{\alpha+\sigma} P_{n}^{\alpha+\sigma,\beta-\sigma}{(\xi)},
\end{eqnarray}
where $0 < \sigma < 1$, $\alpha > -1$, $\beta > -1$, and $P^{\alpha, \, \beta}_{n} (\xi)$ denotes the standard Jacobi polynomials of order $n$ and parameters $\alpha$ and $\beta$ \cite{chen2016gen}. Accordingly, 
\begin{eqnarray}
\label{Eq: 10}
\prescript{RL}{-1}{\mathcal{D}}_{\xi}^{\sigma} P_{n} (\xi)= \frac{\Gamma(n+1)}{\Gamma(n-\sigma+1)}P_{n}^{\, \sigma,-\sigma}{(\xi)}\,(1+\xi)^{-\sigma}
\end{eqnarray}
and
\begin{eqnarray}
\label{Eq: 11}
\prescript{RL}{\xi}{\mathcal{D}}_{1}^{\sigma} P_{n} (\xi)= \frac{\Gamma(n+1)}{\Gamma(n-\sigma+1)}P_{n}^{\, -\sigma,\sigma}{(\xi)}\,(1-\xi)^{-\sigma}, \,
\end{eqnarray}
where $P_{n} (\xi) := P^{\, 0,0}_{n} (\xi)$ represents Legendre polynomial of degree $n$ (see \cite{chen2016gen}). 

Let define the distributed-order derivative as
\begin{equation}
\prescript{D}{}{\mathcal{D}}_{t}^{\phi} f (t,x) :=  \int_{\tau^{min}}^{\tau^{max}} \phi (\tau) \prescript{}{0}{\mathcal{D}}_{t}^{\tau} f (t,x) d\tau,
\end{equation}
where $\alpha \rightarrow \phi(\alpha)$ be a continuous mapping in $[\alpha_{min},\alpha_{max}]$ \cite{kharazmi2016petrov} and $t>0$. 
We note that by choosing the distribution function in the distributed-order derivatives to be the Dirac delta function $\delta(\tau-\tau_0)$, we recover a single (fixed) term fractional derivative, i.e., 
\begin{equation}
\int_{\tau^{min}}^{\tau^{max}} \delta (\tau-\tau_0) \prescript{}{0}{\mathcal{D}}_{t}^{\tau} f (t,x) d\tau = \prescript{}{0}{\mathcal{D}}_{t}^{\tau_0} f (t,x),
\end{equation}
where $\tau_0 \in (\tau_{min},\tau_{max})$.

%
%%%%%%%%%%%%%%%%%%%%%%%%%%%%%%%%%%
\section{Mathematical Formulation}
\label{Sec: General FPDE}
%%%%%%%%%%%%%%%%%%%%%%%%%%%%%%%%%%%%%%%%%%%%
%
We introduce the underlying solution and test spaces along with their proper norms, and also provide some useful lemmas to derive the corresponding bilinear form and thus, prove the well-posedness of the problem.

%%%%%%%%%%%%%%%%%%%%%%%%%%%%%%%%%%
\subsection{\textbf{Mathematical Framework}}
%%%%%%%%%%%%%%%%%%%%%%%%%%%%%%%%%%
%
Recalling the definition of Sobolev space for real $s\geq 0$ from \cite{kharazmi2016petrov,li2009space}, the usual Sobolev space, denoted by $H^s(I)$ on the finite interval $I=(0,T)$, is associated with the norm $\Vert \cdot \Vert_{H^s(I)}$. According to \cite{li2009space,ervin2007variational},  
\begin{equation}
\label{equivalent}
\vert \cdot \vert_{H^{s}(I)} \cong \vert \cdot \vert_{{^l}H^{s}(I)} 
\cong \vert \cdot \vert_{{^r}H^{s}(I)} 
%\cong \vert \cdot \vert_{{}H^{s}(I)}^*,
\end{equation}
where $"\cong"$ denotes equivalence relation, $ \vert \cdot \vert_{{^l}H^{s}(I)} = \Vert \prescript{}{0}{\mathcal{D}}_{t}^{s} (\cdot)\Vert_{L^2(I)}$, and $ \vert \cdot \vert_{{^r}H^{s}(I)} = \Vert \prescript{}{t}{\mathcal{D}}_{T}^{s} (\cdot) \Vert_{L^2(I)}$. Take $\Lambda = (a,b)$. For the real index $\sigma \geq 0$ and $\sigma \neq n-\frac{1}{2}$ on the bounded interval $\Lambda$ the following norms are equivalent \cite{li2010existence}
\begin{equation}
\label{eq14}
\Vert \cdot \Vert_{H^{\sigma}_{}(\Lambda)} \cong \Vert \cdot \Vert_{{^l}H^{\sigma}_{}(\Lambda)} \cong \Vert \cdot \Vert_{{^r}H^{\sigma}_{}(\Lambda)}\cong \vert \cdot \vert_{{}H^{\sigma}(\Lambda)}^*,
\end{equation}
where 
$\Vert \cdot \Vert_{{^l}H^{\sigma}_{}(\Lambda)} = \Big(\Vert \prescript{}{a}{\mathcal{D}}_{x}^{\sigma}\, (\cdot)\Vert_{L^2(\Lambda)}^2+\Vert \cdot \Vert_{L^2(\Lambda)}^2 \Big)^{\frac{1}{2}}$, 
$\Vert \cdot \Vert_{{^r}H^{\sigma}_{}(\Lambda)} = \Big(\Vert \prescript{}{x}{\mathcal{D}}_{b}^{\sigma}\, (\cdot)\Vert_{L^2(\Lambda)}^2+\Vert \cdot \Vert_{L^2(\Lambda)}^2 \Big)^{\frac{1}{2}},$
and $ \vert \cdot \vert_{{^l}H^{\sigma}(I)}^* = \vert (\prescript{}{0}{\mathcal{D}}_{t}^{\sigma}(\cdot),\prescript{}{t}{\mathcal{D}}_{T}^{\sigma}(\cdot))_{I} \vert^{\frac{1}{2}}$.
From Lemma 5.2 in \cite{ervin2007variational}, we have
\begin{equation}
\label{equiv_time_1}
\vert \cdot \vert_{{^l}H^{\sigma}(I)}^*\cong \vert \cdot \vert_{{^l}H^{\sigma}(I)}^{\frac{1}{2}} \, \vert \cdot \vert_{{^r}H^{\sigma}(I)}^{\frac{1}{2}}
= \Vert \prescript{}{0}{\mathcal{D}}_{t}^{\sigma} (\cdot)\Vert_{L^2(I)}^{\frac{1}{2}}\, \Vert \prescript{}{t}{\mathcal{D}}_{T}^{\sigma} (\cdot)\Vert_{L^2(I)}^{\frac{1}{2}}.
\end{equation}
Let $C_{0}^{\infty}(\Lambda)$ represent the space of smooth functions with compact support in $\Lambda$. According to Lemma 3.1 in \cite{samiee2017unified1}, the norms $\Vert \cdot \Vert_{{^l}H^{\sigma}_{}(\Lambda)}$ and $\Vert \cdot \Vert_{{^r}H^{\sigma}_{}(\Lambda)}$ are equivalent to $\Vert \cdot \Vert_{{^c}H^{\sigma}_{}(\Lambda)}$ in space $C^{\infty}_{0}(\Lambda)$, where
\begin{equation}
\Vert \cdot \Vert_{{^c}H^{\sigma}_{}(\Lambda)} = \Big(\Vert \prescript{}{x}{\mathcal{D}}_{b}^{\sigma}\, (\cdot)\Vert_{L^2(\Lambda)}^2+\Vert \prescript{}{a}{\mathcal{D}}_{x}^{\sigma}\, (\cdot)\Vert_{L^2(\Lambda)}^2+\Vert \cdot \Vert_{L^2(\Lambda)}^2 \Big)^{\frac{1}{2}}.
\end{equation}
In the usual Sobolev space, for $u \in {}H^{\sigma}_{}(\Lambda)$ we define
$$\vert u \vert_{{}H^{\sigma}_{}(\Lambda)}^{*}=\vert (\prescript{}{a}{\mathcal{D}}_{x}^{\sigma}\, u, \prescript{}{x}{\mathcal{D}}_{b}^{\sigma}\, v )_{\Lambda}\vert^{\frac{1}{2}}+\vert (\prescript{}{x}{\mathcal{D}}_{b}^{\sigma}\, u, \prescript{}{a}{\mathcal{D}}_{x}^{\sigma}\, v )_{\Lambda}\vert^{\frac{1}{2}}, \quad \forall v \in {}H^{\sigma}_{}(\Lambda),$$
assuming $\underset{u \in H^{\sigma}_{}(\Lambda)}{\sup} \vert (\prescript{}{a}{\mathcal{D}}_{x}^{\sigma}\, u, \prescript{}{x}{\mathcal{D}}_{b}^{\sigma}\, v )_{\Lambda}\vert^{\frac{1}{2}}+\vert (\prescript{}{x}{\mathcal{D}}_{b}^{\sigma}\, u, \prescript{}{a}{\mathcal{D}}_{x}^{\sigma}\, v )_{\Lambda}\vert^{\frac{1}{2}} >0 \quad \forall v \in  H^{\sigma}_{}(\Lambda)$.  Denoted by ${^l}H^{\sigma}_{0}(\Lambda)$ and ${^r}H^{\sigma}_{0}(\Lambda)$ are the closure of $C^{\infty}_0(\Lambda)$ with respect to the norms $\Vert \cdot \Vert_{{^l}H^{\sigma}(\Lambda)}$ and $\Vert \cdot \Vert_{{^r}H^{\sigma} (\Lambda)}$in $\Lambda$, respectively, where $C^{\infty}_0(\Lambda)$ is the space of smooth functions with compact support in $\Lambda$.

%\begin{lem}
%	\label{lemaa32}
%	For $\sigma \geq 0$ and $\sigma \neq n-\frac{1}{2}$, ${^l}H^{\sigma}_{0}(\Lambda)$, ${^r}H^{\sigma}_{0}(\Lambda)$, and ${^c}H^{\sigma}_{0}(\Lambda)$ are equal and their seminorms are equivalent to $\vert \cdot \vert_{{}H^{\sigma}_{}(\Lambda)}^{*}$, where ${^l}H^{\sigma}_{0}(\Lambda)$, ${^r}H^{\sigma}_{0}(\Lambda)$, and ${^c}H^{\sigma}_{0}(\Lambda)$ denotes the closure of $C^{\infty}_{0}(\Lambda)$ with compact support on $\Lambda$ with respect to the norms $\Vert \cdot \Vert_{{^l}H^{\sigma}_{}(\Lambda)}$ and $\Vert \cdot \Vert_{{^r}H^{\sigma}_{}(\Lambda)}$.
%\end{lem}
%\begin{proof}
%See \cite{samiee2017unified1}.
%\end{proof}
\vspace{0.1cm}

Recalling from \cite{kharazmi2016petrov}, ${^\mathfrak{D}}H^{\varphi}(\mathbb{R})$ represents the \textit{distributed Sobolev} space on $\mathbb{R}$ , which is associated with the following norm 
\begin{align}
\label{Eq: Norm dis Frac Sobolev on R}
\Vert \cdot \Vert_{{^\mathfrak{D}}H^{\varphi}(\mathbb{R})}  =  
\left(
\int_{\tau^{min}}^{\tau^{max}} \varphi(\tau) \,\, \Vert \, (1+\vert \omega \vert^2 )^{\frac{\tau}{2}}\mathcal{F}(\cdot)(\omega) \Vert^2_{L^2(\mathbb{R})} \,\, d\tau
\right)^{\frac{1}{2}},
\end{align}
where $0<\varphi(\tau) \in L^1(\,[\tau^{min},\tau^{max}]\,)$, $0 \leq \tau^{min} < \tau^{max} $.
Subsequently, we denote by ${^\mathfrak{D}}H^{\varphi}(I)$ the \textit{distributed Sobolev} space on the bounded open interval $I=(0,T)$, which is defined as 
$
{^\mathfrak{D}}H^{\varphi}(I) = \{ v \in L^2(I) \vert \,\,\, \exists \tilde{v} \in {^\mathfrak{D}}H^{\varphi}(\mathbb{R})\,\,\, s.t.\,\,\, \tilde{v}|_{I} = v \}$
%
%$\varphi \in L^1(\,[\alpha_{1},\alpha_{2}]\,), \, 0 \leq \alpha_{1} < \alpha_{2} \}$ is a non-negative finite functio
%
with the the equivalent norms $\Vert \cdot \Vert_{{^{l,\mathfrak{D}}}H^{\varphi}(I)}$ and $\Vert \cdot \Vert_{{^{r,\mathfrak{D}}}H^{\varphi}(I)}$ in \cite{kharazmi2016petrov}, where
$$
\label{Eq: LeftEquivallent dis. NormFrac Sobolev on Interval}
\Vert \cdot \Vert_{{^{l,\mathfrak{D}}}H^{\varphi}(I)}  =  
\left(
\Vert \cdot \Vert^2_{L^2(I)} + 
\int_{\tau^{min}}^{\tau^{max}} \varphi(\tau) \,\, \Vert \, \prescript{}{0}{\mathcal{D}}_{t}^{\tau}(\cdot)  \Vert^2_{L^2(I)} \,\, d\tau 
\right)^{\frac{1}{2}}, $$ and
$$ 
\Vert \cdot \Vert_{^{r,\mathfrak{D}}H^{\varphi}(I)}  = \left(
\Vert \cdot \Vert^2_{L^2(I)} + \int_{\tau^{min}}^{\tau^{max}} \varphi(\tau) \,\, \Vert \, \prescript{}{t}{\mathcal{D}}_{T}^{\tau}(\cdot)  \Vert^2_{L^2(I)} \,\, d\tau 
\right)^{\frac{1}{2}}.$$

In each realization of a physical process (e.g., sub- or super-diffusion) the distribution function $\varphi(\tau)$ can be obtained from experimental observations, while the theoretical setting of the problem remains invariant. More importantly, choice of  \textit{distributed Sobolev} space and the associated norms provide a sharper estimate for the accuracy of the proposed PG method.

Let $\Lambda_1 = (a_1,b_1)$, $\Lambda_i = (a_i,b_i) \times \Lambda_{i-1}$ for $i=2,\cdots,d$. We define $\mathcal{X}_1 = {^{\mathfrak{D}}}H^{\rho_1}(\Lambda_1)$ with the associated norm $ \Vert \cdot \Vert_{{^{\mathfrak{D}}}H^{\rho_1}_{}(\Lambda_1)}$, where 
\begin{equation}
\label{equivNormDis}
\Vert \cdot \Vert_{{^{\mathfrak{D}}}H^{\rho_1}_{}(\Lambda_1)}=\left(
\Vert \cdot \Vert^2_{L^2(I)} + 
\int_{\nu^{min}_1}^{\nu^{max}_1} \rho_1(\nu_1) \,\left( \Vert \, \prescript{}{a_1}{\mathcal{D}}_{x_1}^{\nu_1}(\cdot)  \Vert^2_{L^2(\Lambda_1)} +\Vert \prescript{}{x_1}{\mathcal{D}}_{b_1}^{\nu_1}(\cdot)  \Vert^2_{L^2(\Lambda_1)}\right) \, d\nu_1 
\right)^{\frac{1}{2}}.
\end{equation}
%Note, ${}H^{\nu_1}_{0}(\Lambda_1)$ denotes the closure of $C^{\infty}_{0}(\Lambda_1)$, which has the compact support in $\Lambda_1$, with respect to the norm $\Vert \cdot \Vert_{{^c}H^{\sigma}_{}(\Lambda_1)}$. 
Subsequently, we construct $\mathcal{X}_d$ such that
\begin{eqnarray}
\mathcal{X}_2 &=& {^{\mathfrak{D}}}H^{\rho_2} \Big((a_2,b_2); L^2(\Lambda_1) \Big) \cap L^2((a_2,b_2); \mathcal{X}_1),
\nonumber
\\
&\vdots&
\nonumber
\\
\mathcal{X}_d &=& {^{\mathfrak{D}}}H^{\rho_d} \Big((a_d,b_d); L^2(\Lambda_{d-1}) \Big) \cap L^2((a_d,b_d); \mathcal{X}_{d-1}),
\end{eqnarray}
associated with the norm
\begin{equation}
\label{norm_Xd}
\Vert \cdot \Vert_{\mathcal{X}_d} = \bigg{\{} \Vert \cdot \Vert_{{^{\mathfrak{D}}}H^{\rho_d} \Big((a_d,b_d); L^2(\Lambda_{d-1}) \Big)}^2 + \Vert \cdot \Vert_{ L^2\Big((a_d,b_d); \mathcal{X}_{d-1}\Big)}^2 \bigg{\}}^{\frac{1}{2}}.
\end{equation}
\begin{lem}
	\label{norm_221}
	Let $\nu_i > 0$ and $\nu_i \neq n-\frac{1}{2}$ for $i=1,\cdots,d$. Then
	\begin{equation}
	\label{norm_Xd_2}
	\Vert \cdot \Vert_{\mathcal{X}_d} \cong \bigg{\{}  \sum_{i=1}^{d}  \int_{\nu_i^{min}}^{\nu_i^{max}} \rho_i(\nu_i) \, \Big( \Vert \prescript{}{x_i}{\mathcal{D}}_{b_i}^{\nu_i}\, (\cdot)\Vert_{L^2(\Lambda_d)}^2+\Vert \prescript{}{a_i}{\mathcal{D}}_{x_i}^{\nu_i}\, (\cdot)\Vert_{L^2(\Lambda_d)}^2 \Big)    d\nu_i  + \Vert  \cdot \Vert_{L^2(\Lambda_d)}^2 \bigg{\}}^{\frac{1}{2}}.
	\end{equation}
\end{lem}
\begin{proof}
	Considering \eqref{equivNormDis}, $\mathcal{X}_1$ is endowed with $\Vert \cdot \Vert_{\mathcal{X}_1}\cong \Vert \cdot \Vert_{{^{\mathfrak{D}}}H^{\rho_1} (\Lambda_1)}$. Next, $\mathcal{X}_2$ is associated with 
	$\Vert \cdot \Vert_{\mathcal{X}_2}=\{  \Vert \cdot \Vert_{{^c}H^{\rho_2} \Big((a_2,b_2); L^2(\Lambda_{1}) \Big)}^2+ \Vert \cdot \Vert_{ L^2\Big((a_2,b_2); \mathcal{X}_{1}\Big)}^2   \}, $
	where
	\begin{eqnarray}
	&&\Vert u \Vert_{{^{\mathfrak{D}}}H^{\rho_2} \Big((a_2,b_2); L^2(\Lambda_{1}) \Big)}^2 
	\nonumber
	\\
	&&= \int_{\nu_2^{min}}^{\nu_2^{max}} \rho_2(\nu_2)\, \int_{a_1}^{b_1}\, \Big( \int_{a_2}^{b_2}\,  \vert \prescript{}{a_2}{\mathcal{D}}_{x_2}^{\nu_2} u  \vert^2  \,  dx_2 + \int_{a_2}^{b_2}\,  \vert \prescript{}{x_2}{\mathcal{D}}_{b_2}^{\nu_2} u \vert^2  \,  dx_2 + \int_{a_2}^{b_2}\,  \vert  u \vert^2  \,  dx_2 \Big) \,dx_1 \, d\nu_2
	\nonumber
	\\
	&&=
	\int_{\nu_2^{min}}^{\nu_2^{max}} \rho_2(\nu_2)\, \Big(\Vert \prescript{}{x_2}{\mathcal{D}}_{b_2}^{\nu_2}\, (u)\Vert_{L^2(\Lambda_2)}^2+\Vert \prescript{}{a_2}{\mathcal{D}}_{x_2}^{\nu_2}\, (u)\Vert_{L^2(\Lambda_2)}^2 \Big) d\nu_2 +\Vert u \Vert_{L^2(\Lambda_2)}^2
	\nonumber
	\end{eqnarray}
	and
	\begin{eqnarray}
	&&\Vert u \Vert_{L^2\Big((a_2,b_2); \mathcal{X}_{1}\Big)}^2
	\nonumber
	\\
	&&=\int_{\nu_1^{min}}^{\nu_1^{max}} \rho_1(\nu_1)\, \int_{a_2}^{b_2}\, \Big( \int_{a_1}^{b_1}\,  \vert \prescript{}{a_1}{\mathcal{D}}_{x_1}^{\nu_1} u  \vert^2  \,  dx_1 + \int_{a_1}^{b_1}\,  \vert \prescript{}{x_1}{\mathcal{D}}_{b_1}^{\nu_1} u \vert^2  \,  dx_1 + \int_{a_1}^{b_1}\,  \vert u \vert^2  \,  dx_1 \Big) \,dx_2\, d\nu_1
	\nonumber
	\\
	&&=
	\int_{\nu_1^{min}}^{\nu_1^{max}} \rho_1(\nu_1)\, \Big(\Vert \prescript{}{x_1}{\mathcal{D}}_{b_1}^{\nu_1}\, u\Vert_{L^2(\Lambda_2)}^2+\Vert \prescript{}{a_1}{\mathcal{D}}_{x_1}^{\nu_1}\, u\Vert_{L^2(\Lambda_2)}^2 \Big) d\nu_1+\Vert u \Vert_{L^2(\Lambda_2)}^2.
	\nonumber
	\end{eqnarray}
	Let assume that
	\begin{equation}
	\label{norm_Xd_d-1}
	\nonumber
	\Vert \cdot \Vert_{\mathcal{X}_{d-1}} \cong \bigg{\{}  \sum_{i=1}^{d-1}  \int_{\nu_i^{min}}^{\nu_i^{max}} \rho_i(\nu_i) \, \Big( \Vert \prescript{}{x_i}{\mathcal{D}}_{b_i}^{\nu_i}\, (\cdot)\Vert_{L^2(\Lambda_{d-1})}^2+\Vert \prescript{}{a_i}{\mathcal{D}}_{x_i}^{\nu_i}\, (\cdot)\Vert_{L^2(\Lambda_{d-1})}^2 \Big)    d\nu_i  + \Vert  \cdot \Vert_{L^2(\Lambda_{d-1})}^2 \bigg{\}}^{\frac{1}{2}}.
	\end{equation}
	Then,
	\begin{eqnarray}
	\label{11141}
	&&\Vert u \Vert_{{^{\mathfrak{D}}}H^{\rho_d}  \Big((a_d,b_d); L^2(\Lambda_{d-1}) \Big)}^2 
	\nonumber
	\\
	&&= 
	\int^{}_{\Lambda_{d-1}}\Big(\int_{a_d}^{b_d}\,  \vert  u \vert^2  \,  dx_d \Big)d\Lambda_{d-1} 
		\nonumber
	\\
	&&
	+ \int^{}_{\Lambda_{d-1}}\Big( \int_{a_d}^{b_d}\int_{\nu_d^{min}}^{\nu_d^{max}} \rho_d(\nu_d) \big(     \vert \prescript{}{a_d}{\mathcal{D}}_{x_d}^{\nu_d} u  \vert^2  + \vert \prescript{}{x_d}{\mathcal{D}}_{b_d}^{\nu_d} u \vert^2 \big)  d\nu_d  \,dx_d\Big)d\Lambda_{d-1}
	\nonumber
	\\
	&&=
	\int_{\nu_d^{min}}^{\nu_d^{max}} \rho_d(\nu_d)\, \Big(\Vert \prescript{}{x_d}{\mathcal{D}}_{b_d}^{\nu_d}\, (u)\Vert_{L^2(\Lambda_d)}^2+\Vert \prescript{}{a_d}{\mathcal{D}}_{x_d}^{\nu_d}\, (u)\Vert_{L^2(\Lambda_d)}^2 \Big) d\nu_d +\Vert u \Vert_{L^2(\Lambda_d)}^2,
	\end{eqnarray}
	and
	\begin{align}
	\label{11142}
	& \Vert u \Vert_{L^2\Big((a_d,b_d); \mathcal{X}_{d-1}\Big)}^2 
	\nonumber
	\\
	&= \int_{a_d}^{b_d} \int_{\Lambda_{d-1}}^{} \Big( \sum_{i=1}^{d-1}\int_{\nu_i^{min}}^{\nu_i^{max}} \rho_i(\nu_i)\,   \big(\vert \prescript{}{a_i}{\mathcal{D}}_{x_i}^{\nu_i} u  \vert^2 +  \vert \prescript{}{x_i}{\mathcal{D}}_{b_i}^{\nu_i} u \vert^2  \big) d\nu_{i}   \Big)d\Lambda_{d-1}\, dx_d 
	\nonumber
	\\
	& \quad +\int_{a_d}^{b_d} \int_{\Lambda_{d-1}}^{} \vert u \vert^2d\Lambda_{d-1}\, dx_d
	\nonumber
	\\
	&	=
	\sum_{i=1}^{d-1}\int_{\nu_i^{min}}^{\nu_i^{max}} \rho_i(\nu_i)\,  \Big( \Vert \prescript{}{x_i}{\mathcal{D}}_{b_i}^{\nu_i}\, u\Vert_{L^2(\Lambda_d)}^2+\Vert \prescript{}{a_i}{\mathcal{D}}_{x_i}^{\nu_i}\, u\Vert_{L^2(\Lambda_d)}^2 \Big)d\nu_i+\Vert u \Vert_{L^2(\Lambda_d)}^2.
	\end{align}
	Therefore, \eqref{norm_Xd_2} arises from \eqref{11141} and \eqref{11142} and the proof is complete.
\end{proof}

The following assumptions allow us to prove the uniqueness of the bilinear form.
\begin{assum}
	\label{assum 1}
	For $u \in \mathcal{X}_d$
	\begin{align*}
	\underset{u \in \mathcal{X}_d}{\sup} \int_{\nu_i^{min}}^{\nu_i^{max}} \rho_i(\nu_i)\,  \Big(\vert \big(\prescript{}{a_i}{\mathcal{D}}_{x_i}^{\nu_i} u, \prescript{}{x_i}{\mathcal{D}}_{b_i}^{\nu_i} v\big)_{\Lambda_d}\vert+\vert \big(\prescript{}{x_i}{\mathcal{D}}_{b_i}^{\nu_i} u, \prescript{}{a_i}{\mathcal{D}}_{x_i}^{\nu_i} v\big)_{\Lambda_d} \vert \Big)d\nu_i >0, \quad \forall v \in \mathcal{X}_d
	\end{align*}
	when $i=1,\cdots,d$, and $\Lambda_{d}^i=\prod_{\underset{j\neq i}{j=1}}^{d}(a_j,b_j)$.
	%{^c}H^{\nu_i}_0 \Big((a_i,b_i); L^2(\Lambda_{d}^i) \Big)
\end{assum}
\begin{assum}
	\label{assum 2}
	For $u \in {^{l,\mathfrak{D}}}H^{\varphi}(I; L^2(\Lambda_d))$
	$\underset{0 \neq u \in\prescript{l,\mathfrak{D}}{}H^{\varphi}(I; L^2(\Lambda_d))}{\sup}\int_{\tau^{min}}^{\tau^{max}}\varphi(\tau)\,\vert( \prescript{}{0}{\mathcal{D}}_{t}^{\tau} u, \prescript{}{t}{\mathcal{D}}_{T}^{\tau} v)_{\Omega} \vert d\tau >0$ $\forall v \in {^{r,\mathfrak{D}}}H^{\varphi}(I; L^2(\Lambda_d))$.
\end{assum}

In Lemma 3.3 in \cite{samiee2017unified1}, it is shown that if $\,1<2\nu_i<2$ for $i=1,\cdots,d$ and $u,v \in  \mathcal{X}_d$, then $\big(\prescript{}{x_i}{\mathcal{D}}_{b_i}^{2\nu_i} u, v\big)_{\Lambda_d}=\big(\prescript{}{x_i}{\mathcal{D}}_{b_i}^{\nu_i} u, \prescript{}{a_i}{\mathcal{D}}_{x_i}^{\nu_i} v\big)_{\Lambda_d},$ and
$\big(\prescript{}{a_i}{\mathcal{D}}_{x_i}^{2\nu_i} u, v\big)_{\Lambda_d}=\big(\prescript{}{a_i}{\mathcal{D}}_{x_i}^{\nu_i} u, \prescript{}{x_i}{\mathcal{D}}_{b_i}^{\nu_i} v\big)_{\Lambda_d}.$ Consequently, we derive
\begin{equation}
\label{ddd1}
\int_{\nu_i^{min}}^{\nu_i^{max}} \rho_i(\nu_i)\, \big(\prescript{}{x_i}{\mathcal{D}}_{b_i}^{2\nu_i} u, v\big)_{\Lambda_d}\, d\nu_i=\int_{\nu_i^{min}}^{\nu_i^{max}} \rho_i(\nu_i)\,\big(\prescript{}{x_i}{\mathcal{D}}_{b_i}^{\nu_i} u, \prescript{}{a_i}{\mathcal{D}}_{x_i}^{\nu_i} v\big)_{\Lambda_d}\, d\nu_i
\end{equation}
and
\begin{equation}
\label{ddd2}
\int_{\nu_i^{min}}^{\nu_i^{max}} \rho_i(\nu_i)\, \big(\prescript{}{a_i}{\mathcal{D}}_{x_i}^{2\nu_i} u, v\big)_{\Lambda_d}\, d\nu_i=\int_{\nu_i^{min}}^{\nu_i^{max}} \rho_i(\nu_i)\,\big(\prescript{}{a_i}{\mathcal{D}}_{x_i}^{\nu_i} u, \prescript{}{x_i}{\mathcal{D}}_{b_i}^{\nu_i} v\big)_{\Lambda_d}\, d\nu_i.
\end{equation}
Additionally, in the light of Lemma 3.2 in \cite{samiee2017unified1}, we have
\begin{align}
\label{equiv_space}
& \int_{\nu_i^{min}}^{\nu_i^{max}} \rho_i(\nu_i)\,\Big( \vert \big(\prescript{}{a_i}{\mathcal{D}}_{x_i}^{\nu_i} u, \prescript{}{x_i}{\mathcal{D}}_{b_i}^{\nu_i} v\big)_{\Lambda_d} \vert + \vert \big(\prescript{}{x_i}{\mathcal{D}}_{b_i}^{\nu_i} u, \prescript{}{a_i}{\mathcal{D}}_{x_i}^{\nu_i} v\big)_{\Lambda_d} \vert\Big) d\nu_i 
\\ \nonumber
&\cong 
\,\, \vert u \vert_{{^{\mathfrak{D}}}H^{\rho_i} \Big((a_i,b_i); L^2(\Lambda_{d}^i) \Big)} \, \vert v \vert_{{^{\mathfrak{D}}}H^{\rho_i} \Big((a_i,b_i); L^2(\Lambda_{d}^i) \Big)},
\end{align}
%and similarly
%\begin{equation}
%\label{equiv_space2}
%\int_{\nu_i^{min}}^{\nu_i^{max}} \rho_i(\nu_i)\, \vert \big(\prescript{}{x_i}{\mathcal{D}}_{b_i}^{\nu_i} u, \prescript{}{a_i}{\mathcal{D}}_{x_i}^{\nu_i} v\big)_{\Lambda_d} \vert d\nu_i \cong  \vert u \vert_{{^c}H^{\rho_i} \Big((a_i,b_i); L^2(\Lambda_{d}^i) \Big)} \, \vert v \vert_{{^c}H^{\rho_i} \Big((a_i,b_i); L^2(\Lambda_{d}^i) \Big)}
%\end{equation}
%
for $i=1,\cdots,d$, where Assumption \ref{assum 1} holds.

Next, we study the property of the fractional time-derivative in the following lemmas. 
\begin{lem}
	\label{lemma_ehsan}
	If  $0<2\tau^{min}<2\tau^{max}<1$ $(1 <2\tau^{min}<2\tau^{max}< 2)$ and $u,v \, \in \, {^{l,\mathfrak{D}}}H^{\varphi}(I)$, when $u\vert_{t=0}(=\frac{du}{dt}\vert_{t=0})=0$, then 
	\begin{equation}
	\label{eq-ehsan}
	\int_{\tau^{min}}^{\tau^{max}}\varphi(\tau)\,\big(\prescript{}{0}{\mathcal{D}}_{t}^{2\tau} u, v\big)_I d\tau=\int_{\tau^{min}}^{\tau^{max}}\varphi(\tau)\,\big(\prescript{}{0}{\mathcal{D}}_{t}^{\tau} u, \prescript{}{t}{\mathcal{D}}_{T}^{\tau} v\big)_I\, d\tau,
	\end{equation}
	where $I=(0,\, T)$, $0<\varphi(\tau)\in L^1\Big([\tau^{min},\tau^{max}]\Big)$. 
	%and $\varphi(\tau)<s$, when $s>0$ is a finite real value.
\end{lem}
\begin{proof}
	It follows from \cite{kharazmi2016petrov} that for $u,v \in H^{\tau}(I)$, when $u\vert_{t=0}(=\frac{du}{dt}\vert_{t=0})=0$ and $v\vert_{t=T}(=\frac{dv}{dt}\vert_{t=T})=0$, we have
	\begin{equation}
	\label{eq-ehsan2}
	\big(\prescript{}{0}{\mathcal{D}}_{t}^{2\tau} u, v\big)_I=\big(\prescript{}{0}{\mathcal{D}}_{t}^{\tau} u, \prescript{}{t}{\mathcal{D}}_{T}^{\tau} v\big)_I.
	\end{equation} 
	Then \eqref{eq-ehsan} arises from \eqref{eq-ehsan2}.
\end{proof}
Let $0<2\tau^{min}<2\tau^{max}<1$ $(1<2\tau^{min}<2\tau^{max}<2),$ and $\Omega=I\times \Lambda_d$, where $I=(0,T)$ and $\Lambda_d=\prod_{i=1}^{d}(a_i,b_i)$. We define
\begin{align}
& {^{l,\mathfrak{D}}}H^{\varphi} \Big(I; L^2(\Lambda_d) \Big) 
\\ \nonumber
&:= 
\Big{\{} u \,|\, \Vert u(t,\cdot) \Vert_{L^2(\Lambda_d)} \in {^{l,\mathfrak{D}}}H^{\varphi}(I), u\vert_{t=0}(=\frac{du}{dt}\vert_{t=0})=u\vert_{x_i=a_i}=u\vert_{x_i=b_i}=0,\, i=1,\cdots,d  \Big{\}},
\end{align}
which is endowed with the norm $\Vert \cdot \Vert_{{^{l,\mathfrak{D}}}H^{\varphi} \Big(I; L^2(\Lambda_d) \Big)}$, where
we have
\begin{align}
\label{norm_222}
\Vert u \Vert_{{^{l,\mathfrak{D}}}H^{\varphi}(I; L^2(\Lambda_d))} = &\Big{\Vert} \, \Vert u(t,\cdot) \Vert_{L^2(\Lambda_d)}\, \Big{\Vert}_{{^{l,\mathfrak{D}}}H^{\varphi}(I)}
\nonumber
%\\
%=&
%\bigg{\{}	\int_{\tau^{min}}^{\tau^{max}}\varphi(\tau)\, \int_{0}^{T}\bigg( \big{(}\int_{\Lambda_d}^{}  \vert \prescript{}{0}{\mathcal{D}}_{t}^{\tau} u \vert^2  \, d\Lambda_d \big{)}^{\frac{1}{2}} \bigg)^{2}   \,dt\, d\tau + \int_{0}^{T} \int_{\Lambda_d}^{}  \vert u \vert^2  \, d\Lambda_d   \,dt\bigg{\}}^{\frac{1}{2}} 
%\nonumber
\\
=& 
%\bigg{\{} \int_{0}^{T} \int_{\Lambda_d}^{}  \vert \prescript{}{0}{\mathcal{D}}_{t}^{\tau} u \vert^2  \, d\Lambda_d dt \bigg{\}}^{\frac{1}{2}}= 
\Big(	\int_{\tau^{min}}^{\tau^{max}}\varphi(\tau)\,\Vert \prescript{}{0}{\mathcal{D}}_{t}^{\tau}\, (u)\Vert_{L^2(\Omega)}^2\, d\tau + \Vert u\Vert_{L^2(\Omega)}^2 \Big)^{\frac{1}{2}}.
\end{align}
\noindent Similarly, we define
\begin{align}
& {^{r,\mathfrak{D}}}H^{\varphi} \Big(I; L^2(\Lambda_d) \Big) 
\\
\nonumber 
& := \Big{\{} v \,|\, \Vert v(t,\cdot) \Vert_{L^2(\Lambda_d)} \in {^{r,\mathfrak{D}}}H^{\varphi}(I), v\vert_{t=T}(=\frac{dv}{dt}\vert_{t=0})=v\vert_{x_i=a_i}=v\vert_{x_i=b_i}=0,\, i =1,\cdots,d  \Big{\}},
\end{align}
which is equipped with the norm $\Vert \cdot \Vert_{{^{r,\mathfrak{D}}}H^{\varphi}(I; L^2(\Lambda_d))}$. Following \eqref{norm_222},
\begin{eqnarray}
\Vert v \Vert_{{^{r,\mathfrak{D}}}H^{\varphi}(I; L^2(\Lambda_d))} &=& \Big{\Vert} \, \Vert v(t,\cdot) \Vert_{L^2(\Lambda_d)}\, \Big{\Vert}_{{^{r,\mathfrak{D}}}H^{\varphi}(I)}
\nonumber
\\
&=&
\Big(	\int_{\tau^{min}}^{\tau^{max}}\varphi(\tau)\,\Vert \prescript{}{t}{\mathcal{D}}_{T}^{\tau}\, (v)\Vert_{L^2(\Omega)}^2 d\tau+\Vert v\Vert_{L^2(\Omega)}^2\Big)^{\frac{1}{2}}.
\end{eqnarray}
\begin{lem}
	\label{norm_223}
	For $u\in {^{r,\mathfrak{D}}}H^{\varphi}(I; L^2(\Lambda_d))$ and $0<2\tau^{min}<2\tau^{max}<1$ $(1<2\tau^{min}<2\tau^{max}<2),$ $\int_{\tau^{min}}^{\tau^{max}}\varphi(\tau)\,\vert( \prescript{}{0}{\mathcal{D}}_{t}^{\tau} u, \prescript{}{t}{\mathcal{D}}_{T}^{\tau} v)_{\Omega} \vert\, d\tau \leq \Vert u \Vert_{{^{l,\mathfrak{D}}}H^{\varphi}(I; L^2(\Lambda_d))} \, \Vert v \Vert_{{^{r,\mathfrak{D}}}H^{\varphi}(I; L^2(\Lambda_d))} $ $\forall v \in {^{r,\mathfrak{D}}}H^{\varphi}(I; L^2(\Lambda_d)).$
\end{lem}
\begin{proof} 
	From Lemma 3.6 in \cite{samiee2017unified1} we have
	\begin{eqnarray}
	\vert( \prescript{}{0}{\mathcal{D}}_{t}^{\tau} u, \prescript{}{t}{\mathcal{D}}_{T}^{\tau} v)_{\Omega}\vert
	&\leq & \Big( \Vert \prescript{}{0}{\mathcal{D}}_{t}^{\tau} u \Vert_{L^2(\Omega)}^2+\Vert u \Vert_{L^2(\Omega)}^2\Big)^{\frac{1}{2}} \, \Big( \Vert \prescript{}{t}{\mathcal{D}}_{T}^{\tau} v \Vert_{L^2(\Omega)}^2+ \Vert  v \Vert_{L^2(\Omega)}^2\Big)^{\frac{1}{2}}.
	\end{eqnarray}
	Followingly, by H\"{o}lder inequality 
	\begin{eqnarray}
	&&\int_{\tau^{min}}^{\tau^{max}}\varphi(\tau)\,\vert( \prescript{}{0}{\mathcal{D}}_{t}^{\tau} u, \prescript{}{t}{\mathcal{D}}_{T}^{\tau} v)_{\Omega}\vert d\tau
	\\
	\nonumber
	&&
	=\int_{\tau^{min}}^{\tau^{max}}\varphi(\tau)\,  \int_{\Lambda_d}^{} \int_{0}^{T} \vert \prescript{}{0}{\mathcal{D}}_{t}^{\tau} u \vert\, \vert \prescript{}{t}{\mathcal{D}}_{T}^{\tau} v \vert dt \,d\Lambda_d  \, d\tau
	\\
	\nonumber
	&&\leq \Big( \int_{\tau^{min}}^{\tau^{max}}\int_{\Lambda_d}^{} \int_{0}^{T}  \varphi(\tau) \,\vert \prescript{}{0}{\mathcal{D}}_{t}^{\tau} u \vert^2\, dt d\Lambda_d \Big)^{\frac{1}{2}} \, \Big( \int_{\tau^{min}}^{\tau^{max}} \int_{\Lambda_d}^{} \int_{0}^{T}   \varphi(\tau) \, \vert \prescript{}{t}{\mathcal{D}}_{T}^{\tau} v \vert^2\, dt d\Lambda_d \Big)^{\frac{1}{2}} 
	\nonumber
	%\\
	%\nonumber
	%&&\leq
	%\Big( \int_{\Lambda_d}^{} \int_{0}^{T}  \vert \prescript{}{0}{\mathcal{D}}_{t}^{\tau} u \vert^2\, dt d\Lambda_d + \int_{\Lambda_d}^{} \int_{0}^{T}  \vert u \vert^2\, dt d\Lambda_d \Big)^{\frac{1}{2}} \, \Big( \int_{\Lambda_d}^{} \int_{0}^{T}  \vert \prescript{}{t}{\mathcal{D}}_{T}^{\tau} v \vert^2\, dt d\Lambda_d + \int_{\Lambda_d}^{} \int_{0}^{T}  \vert  v \vert^2\, dt d\Lambda_d \Big)^{\frac{1}{2}} 
	\\
	\nonumber
	&&= \Big(\int_{\tau^{min}}^{\tau^{max}}\varphi(\tau)\, \Vert \prescript{}{0}{\mathcal{D}}_{t}^{\tau} u \Vert_{L^2(\Omega)}^2 d\tau+\Vert u \Vert_{L^2(\Omega)}^2\Big)^{\frac{1}{2}} \, \Big(\int_{\tau^{min}}^{\tau^{max}}\varphi(\tau)\, \Vert \prescript{}{t}{\mathcal{D}}_{T}^{\tau} v \Vert_{L^2(\Omega)}^2 d\tau+ \Vert  v \Vert_{L^2(\Omega)}^2\Big)^{\frac{1}{2}}
	\\
	\nonumber
	&&= \Vert u \Vert_{{^{r,\mathfrak{D}}}H^{\varphi}(I; L^2(\Lambda_d))} \, \Vert v \Vert_{{^{r,\mathfrak{D}}}H^{\varphi}(I; L^2(\Lambda_d))}.
	\end{eqnarray}
\end{proof}
\begin{lem}
	\label{norm_2231}
	For any $u\in {^{l,\mathfrak{D}}}H^{\varphi}(I; L^2(\Lambda_d))$ and $0<2\tau^{min}<2\tau^{max}<1$ $(1<2\tau^{min}<2\tau^{max}<2)$ there exists a constant $c>0$ and independent of $u$ such that 
	\begin{equation}
	\label{Lemma3.7}
	\underset{0 \neq v \in\prescript{r,\mathfrak{D}}{}H^{\varphi}(I; L^2(\Lambda_d))}{\sup} 
	\frac{
		\int_{\tau^{min}}^{\tau^{max}}\varphi(\tau)\,\vert( \prescript{}{0}{\mathcal{D}}_{t}^{\tau} u, \prescript{}{t}{\mathcal{D}}_{T}^{\tau} v)_{\Omega} \vert d\tau}{\vert v \vert_{{^{r,\mathfrak{D}}}H^{\varphi}(I; L^2(\Lambda_d))}} 
	\geq c \vert u \vert_{{^{l,\mathfrak{D}}}H^{\varphi}(I; L^2(\Lambda_d))}, 
	\end{equation}
	under Assumption \ref{assum 2}.
\end{lem}
\begin{proof}
	Following Lemma 2.4 in \cite{duan2017space} and Lemma 3.7 in \cite{samiee2017unified1}, for any $u\in {^{l,\mathfrak{D}}}H^{\varphi}(I; L^2(\Lambda_d))$ let $\mathcal{V}_u= H(t-T) \, \big(u-u\vert_{t=T}\big) $ assuming that $ \int_{\tau^{min}}^{\tau^{max}} \varphi(\tau)\, \vert( \prescript{}{0}{\mathcal{D}}_{t}^{\tau} u, \prescript{}{t}{\mathcal{D}}_{T}^{\tau} u\vert_{t=T})_{\Omega} \vert >0$, where $H(t)$ is the Heaviside function. Evidently, $\mathcal{V}_u \in \prescript{r,\mathfrak{D}}{}H^{\varphi}(I; L^2(\Lambda_d))$.
	From H\"{o}lder inequality, we obtain
	\begin{eqnarray}
	\label{lemma3.7.1}
	&&\Vert \mathcal{V}_u\Vert^2_{{^{r,\mathfrak{D}}}H^{\varphi}(I; L^2(\Lambda_d))}
	\\
	\nonumber
	&&=
	\int_{\tau^{min}}^{\tau^{max}}\varphi(\tau)\,\Vert \prescript{}{t}{\mathcal{D}}_{T}^{\tau}\bigg(H(t-T)\, \big(u-u\vert_{t=T}\big) \bigg) \Vert_{L^2(\Omega)}^2 d\tau
	\\
	\nonumber
	&&=\int_{\tau^{min}}^{\tau^{max}}\varphi(\tau)\,\Vert \prescript{RL}{t}{\mathcal{I}}_{T}^{1-\tau}\, \frac{d}{dt}\bigg(H(t-T)\,\big(u-u\vert_{t=T}\big) \bigg) \Vert_{L^2(\Omega)}^2 d\tau
	\nonumber
	\\
	&&=\int_{\tau^{min}}^{\tau^{max}}\varphi(\tau)\,\Vert \prescript{RL}{t}{\mathcal{I}}_{T}^{1-\tau}\, \bigg(\frac{d\, H(t-T)}{dt}\, \big(u-u\vert_{t=T}\big) + H(t-T) \frac{d\,\big(u-u\vert_{t=T}\big) }{dt}\bigg) \Vert_{L^2(\Omega)}^2 d\tau \, \,
	\nonumber
	\\
	&&=\int_{\tau^{min}}^{\tau^{max}}\varphi(\tau)\,\Vert \prescript{RL}{t}{\mathcal{I}}_{T}^{1-\tau}\, \bigg( H(t-T) \frac{d\,\big(u-u\vert_{t=T}\big) }{dt}\bigg) \Vert_{L^2(\Omega)}^2 d\tau 
	\nonumber
	\\
	&&= \int_{\tau^{min}}^{\tau^{max}}\varphi(\tau)\,\Vert \prescript{}{t}{\mathcal{D}}_{T}^{\tau}\, u \Vert_{L^2(\Omega)}^2 d\tau,
	\nonumber
	\end{eqnarray}
	By \eqref{equivalent}, $\Vert \mathcal{V}_u\Vert^2_{\prescript{r,\mathfrak{D}}{}H^{\varphi}(I; L^2(\Lambda_d))} \cong \int_{\tau^{min}}^{\tau^{max}}\varphi(\tau)\,\Vert \prescript{}{0}{\mathcal{D}}_{t}^{\tau}\, u \Vert_{L^2(\Omega)}^2 d\tau = \Vert u\Vert^2_{\prescript{l,\mathfrak{D}}{}H^{\varphi}(I; L^2(\Lambda_d))}$. Hence, $\Vert \prescript{}{t}{\mathcal{D}}_{T}^{\tau} \mathcal{V}_u \Vert_{L^2(\Omega)}^2 $ $\cong \Vert \prescript{}{0}{\mathcal{D}}_{t}^{\tau} u\Vert^2_{L^2(\Omega)}$. Therefore,
	\begin{eqnarray}
	\label{Lemma3.7.3}
	\int_{\tau^{min}}^{\tau^{max}}\varphi(\tau)\,\vert( \prescript{}{0}{\mathcal{D}}_{t}^{\tau} u, \prescript{}{t}{\mathcal{D}}_{T}^{\tau} \mathcal{V}_u)_{\Omega} \vert   \,d\tau&=&\int_{\tau^{min}}^{\tau^{max}}\varphi(\tau)\,\int_{\Lambda_d}^{} \int_{0}^{T} \vert \prescript{}{0}{\mathcal{D}}_{t}^{\tau} u \vert\,\vert \prescript{}{t}{\mathcal{D}}_{T}^{\tau} \mathcal{V}_u \vert  dt\,d\Lambda_d \,d\tau
	\\ \nonumber
	&
	\geq&
	\tilde{\beta}\int_{\tau^{min}}^{\tau^{max}}\varphi(\tau)\,\int_{\Lambda_d}^{} \int_{0}^{T} \vert \prescript{}{0}{\mathcal{D}}_{t}^{\tau} u \vert^2 dt\,d\Lambda_d \,d\tau
	\\ \nonumber
	&=& \vert u\vert^2_{\prescript{l,\mathfrak{D}}{}H^{\varphi}(I; L^2(\Lambda_d))},
	\end{eqnarray}
	where $\tilde{\beta}>0$ and independent of $u$. Considering \eqref{lemma3.7.1} and \eqref{Lemma3.7.3}, we obtain
	\begin{eqnarray}
	\label{lemma-3.5}
	%\underset{0 \neq v \in\prescript{r}{0}H^{\varphi}(I; L^2(\Lambda_d))}{\sup} 
	\underset{0 \neq v \in\prescript{r,\mathfrak{D}}{}H^{\varphi}(I; L^2(\Lambda_d))}{\sup} 
	\frac{
		\int_{\tau^{min}}^{\tau^{max}}\varphi(\tau)\,\vert( \prescript{}{0}{\mathcal{D}}_{t}^{\tau} u, \prescript{}{t}{\mathcal{D}}_{T}^{\tau} v)_{\Omega} \vert d\tau}{\vert v \vert_{{^{r,\mathfrak{D}}}H^{\varphi}(I; L^2(\Lambda_d))}} 
	&\geq&
	\frac{\int_{\tau^{min}}^{\tau^{max}}\varphi(\tau)\,\vert( \prescript{}{0}{\mathcal{D}}_{t}^{\tau} u, \prescript{}{t}{\mathcal{D}}_{T}^{\tau} \mathcal{V}_u)_{\Omega} \vert d\tau}{\vert \mathcal{V}_u\vert_{\prescript{r,\mathfrak{D}}{}H^{\varphi}(I; L^2(\Lambda_d))}} 
	\nonumber
	\\
	\nonumber
	&\geq & \tilde{\beta} \, \vert u\vert_{\prescript{l,\mathfrak{D}}{}H^{\varphi}(I; L^2(\Lambda_d))}.
	\end{eqnarray}
	%Therefore, \eqref{Lemma3.7} arises from \eqref{lemma-3.5} and the proof is complete.
\end{proof}
\begin{lem}
	\label{lem_generalize_2}
	If  $0<2\tau^{min}<2\tau^{max}<1$ $(1 <2\tau^{min}<2\tau^{max}< 2)$ and $u,v \, \in \, \prescript{l,\mathfrak{D}}{}H^{\varphi}(I; L^2(\Lambda_d))$, then 
	\begin{equation}
	\label{eq-ehsan-2}
	\int_{\tau^{min}}^{\tau^{max}}\varphi(\tau)\,\big(\prescript{}{0}{\mathcal{D}}_{t}^{2\tau} u, v\big)_{\Omega} d\tau=\int_{\tau^{min}}^{\tau^{max}}\varphi(\tau)\,\big(\prescript{}{0}{\mathcal{D}}_{t}^{\tau} u, \prescript{}{t}{\mathcal{D}}_{T}^{\tau} v\big)_{\Omega}\, d\tau,
	\end{equation}
	where $0<\varphi(\tau)\in L^1\Big([\tau^{min},\tau^{max}]\Big)$.
\end{lem}
\begin{proof}
	By Lemma \ref{lemma_ehsan},
	\begin{align}
	&
	\int_{\tau^{min}}^{\tau^{max}}\varphi(\tau)\,\big(\prescript{}{0}{\mathcal{D}}_{t}^{2\tau} u, v\big)_{\Omega} d\tau=
	\int_{\tau^{min}}^{\tau^{max}}\varphi(\tau)\, \int_{\Lambda_d}^{}\int_{0}^{T} \vert \prescript{}{0}{\mathcal{D}}_{t}^{2\tau} u \vert \, \vert  v\vert \,dt\, d\Lambda_d \, d\tau
	\nonumber
	\\
	= & \int_{\Lambda_d}^{} \int_{\tau^{min}}^{\tau^{max}}\varphi(\tau)\,\int_{0}^{T} \vert \prescript{}{0}{\mathcal{D}}_{t}^{\tau} u \vert \, \vert \prescript{}{t}{\mathcal{D}}_{T}^{\tau}  v\vert \,dt\, d\tau\, d\Lambda_d
	= \int_{\tau^{min}}^{\tau^{max}}\varphi(\tau)\,\big(\prescript{}{0}{\mathcal{D}}_{t}^{\tau} u, \prescript{}{t}{\mathcal{D}}_{T}^{\tau} v\big)_{\Omega}\, d\tau.
	\end{align}
\end{proof}
%\begin{rembold}
%\label{lemma3.9}
%Following Lemma \ref{norm_2231}, we can show that
%for any $u\in L^2(\Omega)$ and $2\tau \in (0,1),$ there exists a constant $c>0$ such that 
%\begin{equation}
%\label{Lemma3.91}
%\underset{0 \neq v \in L^2(\Omega)}{\sup}\vert(u,v)_{\Omega} \vert \geq c \Vert u \Vert_{L^2(\Omega)} \, \Vert v \Vert_{L^2(\Omega)}, 
%\end{equation}
%assuming $\underset{u \in L^2(\Omega)}{\sup}\vert(u,v)_{\Omega} \vert >0 \quad \forall v \in L^2(\Omega)$.
%\end{rembold}

\subsection{\textbf{Solution and Test Function Spaces}}
\label{Solution and Test Function Spaces}
Take $0<2\tau^{min}<2\tau^{max}<1$ ($1<2\tau^{min}<2\tau^{max}<2$) and $1<2\nu^{min}_i<2\nu^{max}_i<2$ for $i=1,\cdots,d$. We define the solution space
\begin{equation}
\mathcal{B}^{\varphi,\rho_1,\cdots,\rho_d} (\Omega):= \prescript{l,\mathfrak{D}}{}H^{\varphi}\Big(I; L^2(\Lambda_d) \Big) \cap L^2(I; \mathcal{X}_d),
\end{equation}
associated with the norm
\begin{equation}
\label{def_2222}
\Vert u \Vert_{\mathcal{B}^{\varphi,\rho_1,\cdots,\rho_d}(\Omega)} = \Big{\{}\Vert u \Vert_{\prescript{l,\mathfrak{D}}{}H^{\varphi}(I; L^2(\Lambda_d))}^2 + \Vert u \Vert_{L^2(I; \mathcal{X}_d)}^2 \Big{\}}^{\frac{1}{2}}.
\end{equation}
Considering Lemma \ref{norm_221},
\begin{eqnarray}
\label{norm_2221}
\Vert u \Vert_{L^2(I; \mathcal{X}_d)}&=&  \Big{\Vert} \, \Vert u(t,.) \Vert_{\mathcal{X}_d}\,\Big{\Vert}_{L^2(I)}
\nonumber
\\ &=&\bigg{\{}  \sum_{i=1}^{d}  \int_{\nu_i^{min}}^{\nu_i^{max}} \rho_i(\nu_i) \, \Big( \Vert \prescript{}{x_i}{\mathcal{D}}_{b_i}^{\nu_i}\, (u)\Vert_{L^2(\Omega)}^2+\Vert \prescript{}{a_i}{\mathcal{D}}_{x_i}^{\nu_i}\, (u)\Vert_{L^2(\Omega)}^2 \Big)    d\nu_i  + \Vert  u \Vert_{L^2(\Lambda_d)}^2 \bigg{\}}^{\frac{1}{2}}. \quad \quad
\end{eqnarray}
Therefore, from \eqref{norm_222} and \eqref{norm_2221},
\begin{eqnarray}
\label{norm_22213}
\Vert u \Vert_{\mathcal{B}^{\varphi,\rho_1,\cdots,\rho_d}(\Omega)} &=& \Big{\{}  \Vert u \Vert_{L^2(\Omega)}^2 + \int_{\tau^{min}}^{\tau^{max}}\varphi(\tau)\,\Vert \prescript{}{0}{\mathcal{D}}_{t}^{\tau}\, (u)\Vert_{L^2(\Omega)}^2 d\tau 
\nonumber
\\
&&+ \sum_{i=1}^{d}  \int_{\nu_i^{min}}^{\nu_i^{max}} \rho_i(\nu_i) \, \Big( \Vert \prescript{}{x_i}{\mathcal{D}}_{b_i}^{\nu_i}\, (u)\Vert_{L^2(\Omega)}^2+\Vert \prescript{}{a_i}{\mathcal{D}}_{x_i}^{\nu_i}\, (u)\Vert_{L^2(\Omega)}^2 \Big)    d\nu_i  \Big{\}}^{\frac{1}{2}}. 
\end{eqnarray}
Similarly, we define the test space
\begin{equation}
\mathfrak{B}^{\varphi,\rho_1,\cdots,\rho_d} (\Omega) := \prescript{r,\mathfrak{D}}{}H^{\varphi}\Big(I; L^2(\Lambda_d)\Big) \cap L^2(I; \mathcal{X}_d),
\end{equation}
equipped with the norm
\begin{align}
\Vert v \Vert_{\mathfrak{B}^{\varphi,\rho_1,\cdots,\rho_d}(\Omega)} =& \Big{\{}\Vert v \Vert_{\prescript{r}{}H^{\varphi}(I; L^2(\Lambda_d))}^2  + \Vert v \Vert_{ L^2(I; \mathcal{X}_d)}^2 \Big{\}}^{\frac{1}{2}}
\nonumber
\\
=& \Big{\{}  \Vert v \Vert_{L^2(\Omega)}^2 + \int_{\tau^{min}}^{\tau^{max}}\varphi(\tau)\,\Vert \prescript{}{t}{\mathcal{D}}_{T}^{\tau}\, (v)\Vert_{L^2(\Omega)}^2 d\tau 
\nonumber
\\
&+ \sum_{i=1}^{d}  \int_{\nu_i^{min}}^{\nu_i^{max}} \rho_i(\nu_i) \, \Big( \Vert \prescript{}{x_i}{\mathcal{D}}_{b_i}^{\nu_i}\, (v)\Vert_{L^2(\Omega)}^2+\Vert \prescript{}{a_i}{\mathcal{D}}_{x_i}^{\nu_i}\, (v)\Vert_{L^2(\Omega)}^2 \Big)    d\nu_i  \Big{\}}^{\frac{1}{2}}
\end{align}
by Lemma \eqref{norm_221} and \eqref{norm_222}. Take $\Omega =  I\times \Lambda_d$ for a positive integer  $d$. The Petrov-Galerkin spectral method reads as: find $u\in \mathcal{B}^{\varphi,\rho_1,\cdots,\rho_d} (\Omega)$ such that
\begin{eqnarray}
\label{Eq: general weak form}
a(u,v) = l(v), \quad \forall v \in \mathfrak{B}^{\varphi,\rho_1,\cdots,\rho_d} (\Omega),
\end{eqnarray}
where the functional $l(v)=(f,v)_{\Omega} $ and
\begin{align}
\label{Eq: general weak form_2}
a(u,v)&=\int_{\tau^{min}}^{\tau^{max}} \varphi(\tau) \, (\prescript{}{0}{\mathcal{D}}_{t}^{\tau}\, u, \prescript{}{t}{\mathcal{D}}_{T}^{\tau}\, v )_{\Omega} d\tau  
\\ \nonumber
&+
\sum_{i=1}^{d} \int_{\mu_i^{min}}^{\mu_i^{max}} \varrho_i(\mu_i) \, \Big(  c_{l_i} ( \prescript{}{a_i}{\mathcal{D}}_{x_i}^{\mu_i}\, u,\, \prescript{}{x_i}{\mathcal{D}}_{b_i}^{\mu_i}\, v )_{\Omega}  
+ c_{r_i}  ( \prescript{}{a_i}{\mathcal{D}}_{x_i}^{\mu_i}\, v,\, \prescript{}{x_i}{\mathcal{D}}_{b_i}^{\mu_i}\, u )_{\Omega} \,\Big) d\mu_i
\\ \nonumber
&-\sum_{j=1}^{d}\int_{\nu_j^{min}}^{\nu_j^{max}} \rho_j(\nu_j) \, \Big( k_{l_j}  ( \prescript{}{a_j}{\mathcal{D}}_{x_j}^{\nu_j}\, u,\, \prescript{}{x_j}{\mathcal{D}}_{b_j}^{\nu_j}\, v )_{\Omega}+k_{r_j}  ( \prescript{}{a_j}{\mathcal{D}}_{x_j}^{\nu_j}\, v,\, \prescript{}{x_j}{\mathcal{D}}_{b_j}^{\nu_j}\, u )_{\Omega}\Big) d\nu_j
\\ \nonumber
&+\gamma 
(u,v)_{\Omega}
\end{align}
%is equal to the variational form \eqref{Eq: general weak form} when solution $u$ is smooth enough. 
following \eqref{ddd1}, \eqref{ddd2} and Lemma \ref{lem_generalize_2} and $\gamma, c_{l_i}, \, c_{r_i}, \, \kappa_{l_i},$ and  $\kappa_{r_i}$ are all constant. Besides,  $0<2\tau^{min}<2\tau^{max}<1$ ($1<2\tau^{min}<2\tau^{max}<2$), $1<2\mu^{min}_i<2\mu^{max}_i<2$ and $1<2\nu^{min}_j<2\nu^{max}_j<2$ for $i,j=1,2,\cdots,d$.

\begin{rem}
	%In case $\tau < \frac{1}{2}$,  we assume that $u$ posses enough regularity (see \cite{gorenflo2015time}) to ensure the equivalence between the problem under the strong formulation and the bilinear form.
	In the case $\tau<\tfrac 12$, additional regularity assumptions are required to ensure equivalence between the weak and strong formulations, see [23] for more details.
\end{rem}

$U_N$ and $V_N$ are chosen as the finite-dimensional subspaces of $\mathcal{B}^{\varphi,\rho_1,\cdots,\rho_d}(\Omega)$ and $\mathfrak{B}^{\varphi,\rho_1,\cdots,\rho_d}(\Omega)$, respectively.
Then, the PG scheme reads as: find $u_N \in U_N$ such that
\begin{eqnarray}
\label{Eq: infinit-dim PG method_1111}
a(u_N,v_N) = l(v_N), \quad \forall v \in V_N,
\end{eqnarray} 
where
\begin{align}
\label{Eq: infinit-dim PG method_1122}
a(u_N,v_N)&=
\int_{\tau^{min}}^{\tau^{max}} \varphi(\tau) \, (\prescript{}{0}{\mathcal{D}}_{t}^{\tau}\, u_N, \prescript{}{t}{\mathcal{D}}_{T}^{\tau}\, v_N )_{\Omega} d\tau  
\nonumber\\
\nonumber
&+
\sum_{i=1}^{d} \int_{\mu_i^{min}}^{\mu_i^{max}} \varrho_i(\mu_i) \, \Big{[}  c_{l_i} ( \prescript{}{a_i}{\mathcal{D}}_{x_i}^{\mu_i}\, u_N,\, \prescript{}{x_i}{\mathcal{D}}_{b_i}^{\mu_i}\, v_N )_{\Omega}  
+ c_{r_i}  ( \prescript{}{a_i}{\mathcal{D}}_{x_i}^{\mu_i}\, u_N,\, \prescript{}{x_i}{\mathcal{D}}_{b_i}^{\mu_i}\, v_N )_{\Omega} \,\Big{]} d\mu_i
\\
&-\sum_{j=1}^{d}\int_{\nu_j^{min}}^{\nu_j^{max}} \rho_j(\nu_j) \, \Big{[} k_{l_j}  ( \prescript{}{a_j}{\mathcal{D}}_{x_j}^{\nu_j}\, u_N,\, \prescript{}{x_j}{\mathcal{D}}_{b_j}^{\nu_j}\, v_N )_{\Omega}+k_{r_j}  ( \prescript{}{a_j}{\mathcal{D}}_{x_j}^{\nu_j}\, u_N,\, \prescript{}{x_j}{\mathcal{D}}_{b_j}^{\nu_j}\, v_N )_{\Omega}\Big{]} d\nu_j
\nonumber
\\
&
+\gamma 
(u_N,v_N)_{\Omega}. \quad \quad
\end{align}
Representing $u_N$ as a linear combination of elements in $U_N$,  the finite-dimensional problem \eqref{Eq: infinit-dim PG method_1122} leads to a linear system, known as Lyapunov system, introduced in Section \ref{Sec: method}.

\subsection{\textbf{Well-posedness Analysis}}
\label{Sec: Stability and Convergence of PG}
The following assumption permit us to prove the uniqueness of the weak form of the problem in \eqref{Eq: general weak form} in Theorem \ref{Thm: Well-Posedness_1D}. 
\begin{assum}
	\label{Assum: sup cond}
	For all $v \in \mathfrak{B}^{\varphi,\rho_1,\cdots,\rho_d} (\Omega)$ 
	\begin{align*}
	& \underset{u \in \mathcal{B}^{\varphi,\rho_1,\cdots,\rho_d} (\Omega)}{sup} \int_{\tau^{min}}^{\tau^{max}} \varphi(\tau) \,\vert(\prescript{}{0}{\mathcal{D}}_{t}^{\tau}\, u, \prescript{}{t}{\mathcal{D}}_{T}^{\tau}\, v )_{\Omega} \vert d\tau\,>\,0,
	\\
	& \underset{u \in \mathcal{B}^{\varphi,\rho_1,\cdots,\rho_d} (\Omega)}{sup} \int_{\nu_j^{min}}^{\nu_j^{max}} \rho_j(\nu_j) \,\Big(\vert   ( \prescript{}{a_j}{\mathcal{D}}_{x_j}^{\nu_j}\, u,\, \prescript{}{x_j}{\mathcal{D}}_{b_j}^{\nu_j}\, v )_{\Omega}  \vert
	+ \vert  ( \prescript{}{x_j}{\mathcal{D}}_{b_j}^{\nu_j}\, u,\, \prescript{}{a_j}{\mathcal{D}}_{x_j}^{\nu_j}\, v )_{\Omega}  \vert \Big) d\nu_j\, > \, 0,
	\\
	&
	\underset{u \in \mathcal{B}^{\varphi,\rho_1,\cdots,\rho_d} (\Omega)}{sup} \vert (u,v)_{\Omega}\vert \, > \, 0,
	\end{align*}
	when $j=1,\cdots,d$.
	
\end{assum}

\begin{lem}
	\label{continuity_lem}
	\textbf{(Continuity)} Let Assumption \ref{Assum: sup cond} holds. The bilinear form in \eqref{Eq: general weak form_2} is continuous, i.e., for $u \in \mathcal{B}^{\varphi,\rho_1,\cdots,\rho_d} (\Omega)$,
	\begin{equation}
	\label{continuity_eq}
	\exists \beta > 0, \, \,\, \vert a(u,v)\vert \leq \beta \, \Vert u \Vert_{\mathcal{B}^{\varphi,\rho_1,\cdots,\rho_d}(\Omega)}\Vert v \Vert_{\mathfrak{B}^{\varphi,\rho_1,\cdots,\rho_d}(\Omega)} \,\, \, \forall v \in \mathfrak{B}^{\varphi,\rho_1,\cdots,\rho_d}(\Omega).
	\end{equation}
\end{lem}
\begin{proof}
	It follows from \eqref{equiv_space} and Lemma \ref{norm_223}.
	%With the aid of \eqref{equiv_space} and lemma \ref{norm_223}, we directly conclude \eqref{continuity_eq}.
\end{proof}	
\begin{thm}
	\label{inf_sup_d_lem}
	Let Assumption \ref{Assum: sup cond} holds. The \text{inf-sup} condition of the bilinear form \eqref{Eq: general weak form_2} for any $d \geq 1$ holds with $\beta > 0$, i.e.,
	\begin{eqnarray}
	\label{Eq: inf sup-time_d_well}
	&&%\sup_{v \in \mathfrak{B}^{\tau,\nu_1,\cdots,\nu_d}(\Omega)}
	\underset{0 \neq u \in \mathcal{B}^{\varphi,\rho_1,\cdots,\rho_d} (\Omega)} {\inf} \,\,\underset{0 \neq v \in\mathfrak{B}^{\varphi,\rho_1,\cdots,\rho_d} (\Omega)}{\sup}
	\frac{\vert a(u , v)\vert}{\Vert v\Vert_{\mathfrak{B}^{\varphi,\rho_1,\cdots,\rho_d}(\Omega)}\Vert u\Vert_{\mathcal{B}^{\varphi,\rho_1,\cdots,\rho_d}}(\Omega)} \geq \beta > 0, \quad 
	%\\
	%\nonumber
	%&&\forall u \in \mathcal{B}^{\tau,\nu_1,\cdots,\nu_d} (\Omega) \,\, and \,\, \forall v \in \mathfrak{B}^{\tau,\nu_1,\cdots,\nu_d} (\Omega)
	%%
	\end{eqnarray}
	where $\Omega = I \times \Lambda_d$.
\end{thm}
\begin{proof}
	For $u \in \mathcal{B}^{\varphi,\rho_1,\cdots,\rho_d} (\Omega)$ and $v \in  \mathfrak{B}^{\varphi,\rho_1,\cdots,\rho_d} (\Omega)$ under Assumption \ref{Assum: sup cond},
	\begin{eqnarray}
	\label{equiv-bilinear2}
	\vert a(u,v)\vert &\cong& \vert (u,v)_{\Omega}\vert+\int_{\tau^{min}}^{\tau^{max}} \varphi(\tau) \,\vert(\prescript{}{0}{\mathcal{D}}_{t}^{\tau}\, u, \prescript{}{t}{\mathcal{D}}_{T}^{\tau}\, v )_{\Omega} \vert d\tau
	\nonumber
	\\ 
	&&+
	\sum_{i=1}^{d} \int_{\mu_i^{min}}^{\mu_i^{max}} \rho_i(\mu_i) \, \Big(\vert ( \prescript{}{a_i}{\mathcal{D}}_{x_i}^{\mu_i}\, u,\, \prescript{}{x_i}{\mathcal{D}}_{b_i}^{\mu_i}\, v )_{\Omega} \vert + \vert  ( \prescript{}{x_i}{\mathcal{D}}_{a_i}^{\mu_i}\, u,\, \prescript{}{a_i}{\mathcal{D}}_{x_i}^{\mu_i}\, v )_{\Omega} \vert \Big) d\mu_i
	\nonumber
	\\ 
	&&
	+
	\sum_{j=1}^{d} \int_{\nu_j^{min}}^{\nu_j^{max}} \rho_j(\nu_j) \, \Big(\vert   ( \prescript{}{a_j}{\mathcal{D}}_{x_j}^{\nu_j}\, u,\, \prescript{}{x_j}{\mathcal{D}}_{b_j}^{\nu_j}\, v )_{\Omega}  \vert + \vert  ( \prescript{}{x_j}{\mathcal{D}}_{b_j}^{\nu_j}\, u,\, \prescript{}{a_j}{\mathcal{D}}_{x_j}^{\nu_j}\, v )_{\Omega}  \vert \Big)d\nu_j.  \quad \quad
	\end{eqnarray}
	Following \eqref{equiv_space} and Theorem 4.3 in \cite{samiee2017unified1}, 
	\begin{eqnarray}
	&&\sum_{i=1}^{d} \int_{\nu_i^{min}}^{\nu_i^{max}} \rho_i(\nu_i) \, \Big(\vert (\prescript{}{a_i}{\mathcal{D}}_{x_i}^{\nu_i}\, (u),\prescript{}{x_i}{\mathcal{D}}_{b_i}^{\nu_i}\, (v))_{\Omega} \vert 
	+\vert (\prescript{}{x_i}{\mathcal{D}}_{b_i}^{\nu_i}\, (u),\prescript{}{a_i}{\mathcal{D}}_{x_i}^{\nu_i}\, (v))_{\Omega} \vert \Big)
	\nonumber
	\\
	&&\geq \tilde{C}_1
	\sum_{i=1}^{d} \Big{[} \int_{\nu_i^{min}}^{\nu_i^{max}} \rho_i(\nu_i) \, \Big(\Vert \prescript{}{a_i}{\mathcal{D}}_{x_i}^{\nu_i}\, (u) \Vert_{L^2(\Omega)}\Big) d\nu_i \, \int_{\nu_i^{min}}^{\nu_i^{max}} \rho_i(\nu_i) \,\Big(\Vert \prescript{}{x_i}{\mathcal{D}}_{b_i}^{\nu_i}\, (v)\Vert_{L^2(\Omega)}\Big) d\nu_i 
	\nonumber
	\\
	&&+\int_{\nu_i^{min}}^{\nu_i^{max}} \rho_i(\nu_i)  \, \Big(\Vert \prescript{}{x_i}{\mathcal{D}}_{b_i}^{\nu_i}\, (u) \Vert_{L^2(\Omega)}\Big) d\nu_i \,  \int_{\nu_i^{min}}^{\nu_i^{max}} \rho_i(\nu_i) \, \Big(\Vert \prescript{}{a_i}{\mathcal{D}}_{x_i}^{\nu_i}\, (v)\Vert_{L^2(\Omega)}\Big) d\nu_i  \Big{]}.
	\nonumber
	\end{eqnarray}
	Thus,
	\begin{eqnarray}
	\label{qqqq}
	&&\sum_{i=1}^{d} \int_{\nu_i^{min}}^{\nu_i^{max}} \rho_i(\nu_i) \, \Big(\vert (\prescript{}{a_i}{\mathcal{D}}_{x_i}^{\nu_i}\, (u),\prescript{}{x_i}{\mathcal{D}}_{b_i}^{\nu_i}\, (v))_{\Omega} \vert 
	+\vert (\prescript{}{x_i}{\mathcal{D}}_{b_i}^{\nu_i}\, (u),\prescript{}{a_i}{\mathcal{D}}_{x_i}^{\nu_i}\, (v))_{\Omega} \vert \Big) d\nu_i
	\nonumber
	\\
	&&\geq \tilde{C}_1  \sum_{i=1}^{d} \int_{\nu_i^{min}}^{\nu_i^{max}} \rho_i(\nu_i) \,\Big(\Vert \prescript{}{a_i}{\mathcal{D}}_{x_i}^{\nu_i}\, (u) \Vert_{L^2(\Omega)}  + \Vert \prescript{}{x_i}{\mathcal{D}}_{b_i}^{\nu_i}\, (u) \Vert_{L^2(\Omega)} \Big) d\nu_i
	\nonumber
	\\
	&&\times \sum_{j=1}^{d} \int_{\nu_j^{min}}^{\nu_j^{max}} \rho_j(\nu_j) \,\Big( \Vert \prescript{}{x_j}{\mathcal{D}}_{b_j}^{\nu_j}\, (v)\Vert_{L^2(\Omega)}, + \Vert \prescript{}{a_j}{\mathcal{D}}_{x_j}^{\nu_j}\, (v)\Vert_{L^2(\Omega)}\Big)d\nu_j
	\nonumber
	\\
	&&= \tilde{C}_1  \vert u \vert_{L^2(I; \mathcal{X}_d)} \, \vert v \vert_{L^2(I; \mathcal{X}_d)}.
	\end{eqnarray}
	where $\tilde{C}_1$ is a positive constant and independent of $u$. 
	Considering Lemma \ref{norm_2231}, there exists a positive constant $\tilde{C}_2>0$ and independent of $u$ such that
	\begin{equation}
	\label{qqqq2}
	\underset{0 \neq v \in\mathfrak{B}^{\varphi,\rho_1,\cdots,\rho_d} (\Omega)}{\sup}
	\frac{
		\int_{\tau^{min}}^{\tau^{max}} \varphi(\tau) \,\vert (\prescript{}{0}{\mathcal{D}}_{t}^{\tau}(u),\prescript{}{t}{\mathcal{D}}_{T}^{\tau}(v))_{\Omega} \vert d\tau }{\vert v \vert_{\prescript{r,\mathfrak{D}}{}H^{\varphi}(I;L^2(\Lambda_d))}}
	\geq 
	\tilde{C}_2 \vert u \vert_{\prescript{l,\mathfrak{D}}{}H^{\varphi}(I; L^2(\Lambda_d))}.
	\end{equation}
	Furthermore, for $u \in \mathcal{B}^{\varphi,\rho_1,\cdots,\rho_d} (\Omega)$
	\begin{eqnarray}
	\label{equiv22}
	\underset{0 \neq v \in\mathfrak{B}^{\varphi,\cdots,\rho_d} (\Omega)}{\sup}
	\frac{
		\int_{\tau^{min}}^{\tau^{max}} \varphi(\tau) \,\vert (\prescript{}{0}{\mathcal{D}}_{t}^{\tau}(u),\prescript{}{t}{\mathcal{D}}_{T}^{\tau}(v))_{\Omega} \vert d\tau }{\vert v \vert_{\prescript{r,\mathfrak{D}}{}H^{\varphi}(I;L^2(\Lambda_d))}} \cong \underset{0 \neq v \in\mathfrak{B}^{\varphi,\cdots,\rho_d} (\Omega)}{\sup}
	\frac{
		\int_{\tau^{min}}^{\tau^{max}} \varphi(\tau) \,\vert (\prescript{}{0}{\mathcal{D}}_{t}^{\tau}(u),\prescript{}{t}{\mathcal{D}}_{T}^{\tau}(v))_{\Omega} \vert d\tau }{\vert v \vert_{\mathcal{B}^{\varphi,\rho_1,\cdots,\rho_d} (\Omega)}}
	\end{eqnarray}
	and
	\begin{eqnarray}
	\label{equiv23}
	&&\qquad \underset{0 \neq v \in\mathfrak{B}^{\varphi,\rho_1,\cdots,\rho_d} (\Omega)}{\sup} \frac{
		\sum_{j=1}^{d} \int_{\nu_j^{min}}^{\nu_j^{max}} \rho_j(\nu_j) \, \Big(\vert   ( \prescript{}{a_j}{\mathcal{D}}_{x_j}^{\nu_j}\, u,\, \prescript{}{x_j}{\mathcal{D}}_{b_j}^{\nu_j}\, v )_{\Omega}  \vert + \vert  ( \prescript{}{x_j}{\mathcal{D}}_{b_j}^{\nu_j}\, u,\, \prescript{}{a_j}{\mathcal{D}}_{x_j}^{\nu_j}\, v )_{\Omega}  \vert \Big)d\nu_j}{\Vert v\Vert_{L^2(I; \mathcal{X}_d)}}
	\nonumber
	\\
	&& \qquad \qquad \cong \underset{0 \neq v \in\mathfrak{B}^{\varphi,\rho_1,\cdots,\rho_d} (\Omega)}{\sup}  \frac{
		\sum_{j=1}^{d} \int_{\nu_j^{min}}^{\nu_j^{max}} \rho_j(\nu_j) \, \Big(\vert   ( \prescript{}{a_j}{\mathcal{D}}_{x_j}^{\nu_j}\, u,\, \prescript{}{x_j}{\mathcal{D}}_{b_j}^{\nu_j}\, v )_{\Omega}  \vert + \vert  ( \prescript{}{x_j}{\mathcal{D}}_{b_j}^{\nu_j}\, u,\, \prescript{}{a_j}{\mathcal{D}}_{x_j}^{\nu_j}\, v )_{\Omega}  \vert \Big)d\nu_j}{\Vert v\Vert_{\mathfrak{B}^{\varphi,\rho_1,\cdots,\rho_d}(\Omega)}}. \quad \qquad 
	\end{eqnarray}
	Therefore, from \eqref{qqqq}, \eqref{qqqq2}, \eqref{equiv22}, and \eqref{equiv23} we have
	\begin{eqnarray}
	\label{thm123}
	\underset{0 \neq v \in\mathfrak{B}^{\varphi,\rho_1,\cdots,\rho_d} (\Omega)}{\sup} \frac{\vert a(u,v)\vert}{\Vert v\Vert_{\mathfrak{B}^{\varphi,\rho_1,\cdots,\rho_d}(\Omega)}} 
	&\geq& \bar{\beta} \underset{0 \neq v \in\mathfrak{B}^{\varphi,\rho_1,\cdots,\rho_d} (\Omega)}{\sup}  \frac{\vert (u,v)_{\Omega}\vert+\int_{\tau^{min}}^{\tau^{max}} \varphi(\tau) \,\vert(\prescript{}{0}{\mathcal{D}}_{t}^{\tau}\, u, \prescript{}{t}{\mathcal{D}}_{T}^{\tau}\, v )_{\Omega} \vert d\tau }{\Vert v\Vert_{\mathfrak{B}^{\varphi,\rho_1,\cdots,\rho_d}(\Omega)}}
	\nonumber
	\\
	&&+ \frac{
		\sum_{j=1}^{d} \int_{\nu_j^{min}}^{\nu_j^{max}} \rho_j(\nu_j) \, \Big(\vert   ( \prescript{}{a_j}{\mathcal{D}}_{x_j}^{\nu_j}\, u,\, \prescript{}{x_j}{\mathcal{D}}_{b_j}^{\nu_j}\, v )_{\Omega}  \vert + \vert  ( \prescript{}{x_j}{\mathcal{D}}_{b_j}^{\nu_j}\, u,\, \prescript{}{a_j}{\mathcal{D}}_{x_j}^{\nu_j}\, v )_{\Omega}  \vert \Big)d\nu_j}{\Vert v\Vert_{\mathfrak{B}^{\varphi,\rho_1,\cdots,\rho_d}(\Omega)}}
	\nonumber
	\\
	&
	\geq &
	\bar{\beta} \, \bar{C} \,\Big(\Vert u \Vert_{L^2(\Omega)}+
	\vert u \vert_{\prescript{l,\mathfrak{D}}{}H^{\varphi}(I; L^2(\Lambda_d))} + \vert u \vert_{L^2(I; \mathcal{X}_d)} \Big),
	% \\
	%& \geq \bar{C}\Big(
	% \vert u \vert_{\prescript{l,\mathfrak{D}}{}H^{\varphi}(I; L^2(\Lambda_d))} \, \, \vert v \vert_{\prescript{r,\mathfrak{D}}{}H^{\varphi}(I;L^2(\Lambda_d))} +  \vert u \vert_{L^2(I; \mathcal{X}_d)} \, \vert v \vert_{L^2(I; \mathcal{X}_d)}  +  \Vert u \Vert_{L^2(\Omega)} \,  \Vert v \Vert_{L^2(\Omega)}\Big), \,
	\end{eqnarray}
	where $\bar{C}= min\{\tilde{C}_2, \, \tilde{C}_1 \}$. 
	%Moreover, for $u \in \mathcal{B}^{\varphi,\rho_1,\cdots,\rho_d} (\Omega)$ and $v \in  \mathfrak{B}^{\varphi,\rho_1,\cdots,\rho_d} (\Omega)$
	%%
	%\begin{align}
	%\label{thm124}
	%\nonumber
	%& \Vert u \Vert_{\mathcal{B}^{\tau,\nu_1,\cdots,\nu_d}(\Omega)}\Vert v \Vert_{\mathfrak{B}^{\tau,\nu_1,\cdots,\nu_d}(\Omega)}
	%\\
	%\nonumber
	%&\cong 
	%\Vert u \Vert_{\prescript{r}{}H^{\tau}(I; L^2(\Lambda_d))} \, \, \Vert v \Vert_{\prescript{l}{}H^{\tau}(I;L^2(\Lambda_d))} + \Vert u \Vert_{L^2(I; \mathcal{X}_d)} \, \Vert v \Vert_{L^2(I; \mathcal{X}_d)} 
	%+  \Vert u \Vert_{L^2(\Omega)} \,  \Vert v \Vert_{L^2(\Omega)}.
	%\end{align}
	%%
	Accordingly,
	\begin{equation}
	\label{inequality_eq3}
	\underset{0 \neq u \in \mathcal{B}^{\varphi,\rho_1,\cdots,\rho_d} (\Omega)} {inf} \,\,\underset{0 \neq v \in\mathfrak{B}^{\varphi,\rho_1,\cdots,\rho_d} (\Omega)}{sup} \frac{\vert a(u,v)\vert}{\Vert v \Vert_{\mathfrak{B}^{\varphi,\rho_1,\cdots,\rho_d}(\Omega)}} \geq \beta \, \Vert u \Vert_{\mathcal{B}^{\varphi,\rho_1,\cdots,\rho_d}(\Omega)},
	\end{equation}
	where $\beta=\bar{\beta} \, \bar{C} $ is a positive constant and independent.
\end{proof}
\begin{thm}
	\label{Thm: Well-Posedness_1D}
	\textbf{(Well-Posedness)} For $0<2\tau^{min}<2\tau^{max}<1$ ($1<2\tau^{min}<2\tau^{max}<2$), $1<2\nu^{min}_i<2\nu^{max}_i<2$, and $i=1,\cdots,d$, there exists a unique solution to \eqref{Eq: infinit-dim PG method_1111}, which is continuously dependent on  $f \in \big(\mathcal{B}^{\tau,\nu_1,\cdots,\nu_d}\big)^{\star}(\Omega)$, where $\big(\mathcal{B}^{\tau,\nu_1,\cdots,\nu_d}\big)^{\star}(\Omega)$ is the dual space of $\mathcal{B}^{\tau,\nu_1,\cdots,\nu_d}(\Omega)$.
\end{thm}
\begin{proof}
	In virtue of the generalized Babu\v{s}ka-Lax-Milgram theorem \cite{shen2011spectral}, the well-posedness of the weak form in \eqref{Eq: general weak form} in ($1+d$) dimensions is guaranteed by the continuity and the \textit{inf-sup} condition, which are proven in Lemma \ref{continuity_lem} and Theorem \ref{inf_sup_d_lem}, respectively.
\end{proof}

\section{Petrov Galerkin Method}
\label{Sec: method}
%
%%%%%%%%%%%%%%%%%%%%%%%%%%%%%%%%%%
%
To construct a Petrov-Galerkin spectral method for the finite-dimensional weak form problem in \eqref{Eq: infinit-dim PG method_1111}, we first define the proper finite-dimensional basis/test spaces and then implement the numerical scheme.

%%%%%%%%%%%%%%%%%%%%%%%%%%%%%
\subsection{\textbf{Space of Basis ($U_N$) and Test ($V_N$) Functions}}
\label{Sec: Basis Func TSFA}
%%%%%%%%%%%%%%%%%%%%%%%%%%%%%
% Eq: infinit-dim PG method_1111
As discussed in \cite{samiee2017unified1}, we take the spatial basis, given in the standard domain $ \xi \in [-1,1]$ as $\phi^{}_{m} ( \xi ) = \sigma_m \big{(} P_{m+1} (\xi) - P_{m-1} (\xi)\big{)},\, m=1,2,\cdots$, where $P_{m} (\xi)$ is the Legendre polynomials of order $m$ and $\sigma_m = 2 + (-1)^{m}$. Besides, employing Jacobi \textit{poly-fractonomials} of the first kind \cite{zayernouri2013fractional,zayernouri2015tempered},
the temporal basis functions are given in the standard domain $\eta\in [-1,1]$ as $\psi^{\tau}_n(\eta) = {\sigma}_{n} (1+\eta)^{\tau} P_{n-1}^{-\tau, \tau} (\eta),\, n=1,2,\ldots$.

We also let $\eta(t) = 2t/T -1$ and $\xi_j(s) = 2\frac{s-a_j}{b_j-a_j} -1$ to be temporal and spatial affine mappings from $t\in [0,T]$ and $x_j\in[a_j,b_j]$ to the standard domain $[-1,1]$, respectively. Therefore,
\begin{equation*}
\label{Eq: Trial Space: PG1}
U_N = 
%\begin{cases} 
{\rm span} \Big\{    \Big( \psi^{\,\tau}_n \circ \eta \Big) ( t )
\prod_{j=1}^{d} \Big( \phi^{}_{m_j} \circ \xi_j\Big)  (x_j)
: n = 1,2, \cdots, \mathcal{N}, \, m_j= 1,2, \cdots, \mathcal{M}_j\Big\}.
\end{equation*}
Similarly, we employ Legendre polynomials and Jacobi \textit{polyfractonomials} of second kind in the standard domain to construct the finite dimensional test space as
\begin{align}
\nonumber
%\label{Eq: Test Space: PG}
V_N = {\rm span} \Big\{  \Big(\Psi^{\tau}_r \circ \eta\Big)(t)
\prod_{j=1}^{d} \Big( \Phi^{}_{k_j} \circ \xi_j\Big)(x_j)
: r = 1,2, \cdots, \mathcal{N}, \, k_j= 1,2, \cdots, \mathcal{M}_j\Big\},
\end{align}
where $\Psi^{\tau}_{r}(\eta) = \widetilde{\sigma}_{r} (1-\eta)^{\tau}\, P_{r-1}^{\tau,-\tau} (\eta),\, r=1,2,\cdots$ and $\Phi^{}_{k} (\xi) = \widetilde{\sigma}_{k} \big{(} P_{k+1}^{} (\xi) - P_{k-1}^{} (\xi)\big{)},\, k =1,2,\cdots$. The coefficient $ \widetilde{\sigma}_{k}$ is defined as $ \widetilde{\sigma}_{k} = 2\,(-1)^{k} + 1$. 

Since the univariate basis/test functions belong to the fractional Sobolev spaces (see \cite{zayernouri2013fractional}) and $ 0<\varphi(\tau)\in L^1((\tau^{min},\tau^{max}))$, 
$0<\rho_j(\nu_j)\in L^1((\nu_j^{min},\nu_j^{max}))$ for $j=1,\cdots,d$, then $U_N\subset  \mathcal{B}^{\varphi,\rho_1,\cdots,\rho_d} (\Omega)$ and $V_N\subset  \mathfrak{B}^{\varphi,\rho_1,\cdots,\rho_d} (\Omega)$. Accordingly, we approximate the solution in terms of a linear combination of elements in $U_N$, which satisfies initial and boundary conditions.
%
%%%%%%%%%%%%%%%%%%%%%%%%%%%%%%%%%%%%%%
%
\subsection{\textbf{Implementation of the PG Spectral Method}}
\label{Sec: Implementaiton of PG}
%
%%%%%%%%%%%%%%%%%%%%%%%%%%%%%%%%%%%%%%
%
The solution $u_N$ of \eqref{Eq: infinit-dim PG method_1111} can be represented as
\begin{eqnarray}
\label{Eq: PG expansion}
u_{N}(x,t) = 
\sum_{n=1}^\mathcal{N}
\sum_{m_1=1}^{\mathcal{M}_1}
\cdots 
\sum_{m_d= 1}^{\mathcal{M}_d}
\hat u_{ n,m_1,\cdots,m_d} 
\Big[\psi^{\tau}_n(t)
\prod_{j=1}^{d} \phi^{}_{m_j}(x_j)
\Big]
\end{eqnarray}
in $\Omega$ and also we take $v_N = \Psi^{\tau}_r(t) \prod_{j=1}^{d} \Phi^{}_{k_j}(x_j)$, $r = 1,2, \dots, \mathcal{N}$, $k_j= 1,2, \dots, \mathcal{M}_j$. Accordingly, by replacing $u_N$ and $v_N$ in \eqref{Eq: infinit-dim PG method_1111}, we obtain the following Lyapunov system 
%
%Eq: infinit-dim PG method_1111Eq: infinit-dim PG method_1122
%We enforce the corresponding residual to be $L^2$-orthogonal to $v_N \in V_N$, which leads to the \textcolor{red}{ finite-dimensional} weak form \eqref{Eq: Finite-dim PG method}. Specifically, by choosing $v_N = \Psi^{\,\tau}_r(t) \prod_{j=1}^{d} \Phi^{}_{k_j}(x_j)$, when $r = 1, \dots, \mathcal{N}$ and $k_j= 1, \dots, \mathcal{M}_j$, $j=1,2,\cdots,d$, we get 
%
\begin{align}
\label{Eq: general Lyapunov}
\Big(
S_{\tau}^{\varphi} \otimes M_1 \otimes M_2 \cdots \otimes M_d 
&+
\sum_{j=1}^{d} 
[M_{\tau} \otimes M_1\otimes \cdots   \otimes M_{j-1} \otimes S_{j}^{{Tot}} \otimes M_{j+1}  \cdots \otimes M_d]
\nonumber
\\
&+ 
\gamma M_{\tau}\otimes M_1 \otimes M_2 \cdots \otimes M_d \Big) \mathcal{U}= F,
\end{align} 
in which $\otimes$ represents the Kronecker product, $F$ denotes the multi-dimensional load matrix whose entries are given as
\begin{eqnarray}
\label{Eq: general load matrix}
F_{r,k_1,\cdots, k_d} = \int_{\Omega}^{} f(t,x_1,\cdots,x_d) 
\Big(
\Psi^{\,\tau}_r \circ \eta \Big)(t)
\prod_{j=1}^{d} \Big(\Phi^{}_{k_j} \circ \xi_j\Big)(x_j)\, 
d\Omega, 
\end{eqnarray}
and $S^{\, {Tot}}_{j} = c_{l_j} \times S_{l}^{\varrho_j} + c_{r_j} \times S_{r}^{\varrho_j}-\kappa_{l_j} \times S_{l}^{\rho_j}-\kappa_{r_j} \times S_{r}^{\rho_j}$. The matrices $S_{\tau}^{\varphi}$ and $M_{\tau}$ denote the temporal stiffness and mass matrices, respectively; $ S_{l}^{\varrho_j}$, $ S_{r}^{\varrho_j}$, $ S_{l}^{\rho_j}$, $ S_{r}^{\rho_j}$, and $M_j$ denote the spatial stiffness and mass matrices. The entries of spatial mass matrix $M_j$ are computed analytically, while we employ proper quadrature rules to accurately compute the entries of temporal mass matrix $M_{\tau}$ as discussed in \cite{samiee2017Unified}.
% as discussed in \cite{samiee2017fast}. 
The entries of $S_{\tau}^{\varphi}$ are also computed based on Theorem 3.1 (spectrally/exponentially accurate quadrature rule in $\alpha$-dimension) in \cite{kharazmi2016petrov}. Likewise, we present the computation of $S^{Tot}_{j}$ in Lemma \ref{Thm: Spatial Stiffness Matrix} in Appendix.
\begin{rem}
	The choices of coefficients in the construction of finite dimensional basis/test functions lead to symmetric mass/stiffness matrices, which help formulating the following fast solver.
\end{rem}
%%%%%%%%%%%%%%%%%%%%%%%%%%%%%%%%%%%%%%
%
\subsection{\textbf{Unified Fast FPDE Solver}}
\label{Sec: FastSolver FPDE}
%
%%%%%%%%%%%%%%%%%%%%%%%%%%%%%%%%%%%%%%
%
In order to formulate a closed-form solution to the Lyapunov system \eqref{Eq: general Lyapunov}, we follow \cite{zayernouri2015unified} and develop a fast solver in terms of the generalized eigen-solutions. 

\begin{thm}{\cite{samiee2017Unified}}
	\label{Thm: fast solver}
	Take $\{ \vec{e}_{m_j}^{j}    ,    \lambda^{j}_{m_j}\,  \}_{m_j=1}^{\mathcal{M}_j}$ as the set of general eigen-solutions of the spatial stiffness matrix $S^{Tot}_j$ with respect to the mass matrix $M_{j}$. Besides, let $\{ {\vec{e}_{n}}^{\,\,\tau}    ,    \lambda^{\tau}_{n}\,  \}_{n=1}^{\mathcal{N}}$ be the set of general eigen-solutions of the temporal mass matrix $M_{\tau}$ with respect to the stiffness matrix $S_{\tau}^{\varphi}$. Then the unknown coefficients matrix $\mathcal{U}$ is obtained as
	\begin{equation}
	\label{Eq: thm u expression in terms of k}
	\mathcal{U} = 
	\sum_{n=1}^{\mathcal{N}}
	\,\,
	\sum_{m_1= 1}^{\mathcal{M}_1}
	\cdots 
	\sum_{m_d= 1}^{\mathcal{M}_d}
	\kappa_{ n,m_1,\cdots,\,m_d  } \,
	\,\vec{e}_n^{\,\,\tau}\,
	\otimes
	\,{\vec{e}_{m_1}}^{\,\, 1}\,\,
	\otimes
	\cdots
	\otimes
	\,{\vec{e}_{m_d}}^{\,\, d},
	\end{equation}
	where
	\begin{eqnarray}
	\label{Eq: thm k fraction_1}
	\kappa_{ n,m_1,\cdots,\,m_d  } =  \frac{(\,\vec{e}_n^{\,\,\tau}
		\,{\vec{e}_{m_1}}^{1}
		\cdots
		\,{\vec{e}_{m_d}}^{d}) F}
	{
		\Big[
		(\vec{e}_n^{\,\,\tau^T} S_{\tau}^{\varphi} \vec{e}_n^{\,\,\tau})
		\prod_{j=1}^{d} ({\vec{e}^{j^T}_{m_j}}  M_{j} {\vec{e}_{m_j}^{j}})
		\Big]
		\Lambda_{n,m_1,\cdots,m_d}
	},
	\end{eqnarray}
	and
	\begin{eqnarray}
	\nonumber
	&\Lambda_{n,m_1,\cdots, m_d} = \Big[
	(1+\gamma\,\, 
	\lambda^{\tau}_n)
	+
	\lambda^{\tau}_n
	\sum_{j=1}^{d}
	(
	\lambda^{j}_{m_j}
	)
	\Big].  &
	\end{eqnarray}
	
\end{thm}
%%PPPPPPPPPPPPPPPPPPPPPPPPPP
\begin{rem}
	The naive computation of all entries in \eqref{Eq: thm k fraction_1} leads to a computational complexity of  $O(\mathcal{N}^{2+2d})$, including construction of stiffness and mass matrices. By performing \textit{sum-factorization} \cite{zayernouri2015unified}, the operator counts can be reduced to $O(\mathcal{N}^{2+d})$. 
	
	%In other words, \textit{sum-factorization} technique helps us reduce the computational complexity to $O(\mathcal{N}^{3+d})$.
\end{rem}

\section{Stability and Error Analysis}
\label{Sec: error analysis of PG}
The following theorems provide the finite dimensional stability  and error analysis of the proposed scheme, based on the well-posedness analysis from Section \ref{Sec: Stability and Convergence of PG}.

\subsection{\textbf{Stability Analysis}}
\label{Sec: Stability PG}

\begin{thm}
	\label{Thm: inf-sup_3}
	Let Assumption \ref{Assum: sup cond} holds. The Petrov-Gelerkin spectral method for \eqref{Eq: infinit-dim PG method_1122} is stable, i.e., 
	%the discrete $\inf$-$\sup$ condition 
	\begin{eqnarray}
	\label{Eq: inf sup-time}
	&&\underset{0 \neq u_N \in U_N}{inf}\, \, \underset{0 \neq v_N \in V_N}{sup}
	\frac{\vert a(u_N , v_N)\vert}{\Vert v_N\Vert_{\mathfrak{B}^{\varphi,\rho_1,\cdots,\rho_d}(\Omega)}\Vert u_N\Vert_{\mathcal{B}^{\varphi,\rho_1,\cdots,\rho_d}(\Omega)}} \geq \beta > 0, \quad 
	\end{eqnarray}
	holds with $\beta > 0$ and independent of $N$.
	%, where $\underset{0 \neq v_N \in V_N}{\sup} \vert a(u_N , v_N)\vert>0$.
\end{thm}
\begin{proof}
	Regarding $U_N \subset \mathcal{B}^{\varphi,\rho_1,\cdots,\rho_d}(\Omega)$ and $V_N \subset \mathfrak{B}^{\varphi,\rho_1,\cdots,\rho_d}(\Omega)$, \eqref{Eq: inf sup-time} follows directly from Theorem \ref{inf_sup_d_lem}.
\end{proof}
\begin{rem}
	The bilinear form \eqref{Eq: infinit-dim PG method_1122} can be expanded in terms of the basis and test functions to obtain the lower limit of $\beta$, see \cite{zayernouri2015unified,samiee2017Unified}.
\end{rem}

\subsection{\textbf{Error Analysis}}
\label{Sec: error PG}
Denoting by $P_\mathcal{M}(\Lambda)$ the space of all polynomials of degree $\leq \mathcal{M}$ on $\Lambda\subset \mathbb{R}$, $P^{\varphi}_\mathcal{M}(\Lambda):=P_\mathcal{M}(\Lambda) \cap {^\mathfrak{D}}H^{\varphi}(\Lambda)$, where $0<\varphi(\tau)\in L^1((\tau^{min},\tau^{max}))$ and ${^\mathfrak{D}}H^{\varphi}(\Lambda)$ is the  \textit{distributed Sobolev} space associated with the norm $\Vert \cdot \Vert_{{^\mathfrak{D}}H^{\varphi}(\Lambda)}$.
%with compact support on $\Lambda$. 
In this section, we take $I_0=(0,T)$, $I_i=(a_i,b_i)$ for $i=1,...,d$, $\Lambda_i=I_i\times \Lambda_{i-1}$, and $\Lambda_i^j=\prod_{\underset{k\neq j}{k=1}}^{i} I_k$. Besides,  $0<2\tau^{min}<2\tau^{max}<1$ ($1<2\tau^{min}<2\tau^{max}<2$), $1<2\nu^{min}_i<2\nu^{max}_i<2$ for $i=1,\cdots,d$. Where there is no confusion, the symbols $I_i$, $\Lambda_i$, and $\Lambda_i^j$ and the intervals of $(\tau^{min},\tau^{max})$ and $(\nu^{min}_i,\nu^{max}_i)$ will be dropped from the notations.
\begin{thm}
	\label{err_1}
	\cite{maday1990analysis} Let $r_1$ be a real number, where $r_1\neq \mathcal{M}_1 + \frac{1}{2}$, and $1\leq r_1$. There exists a projection operator $\Pi^{\nu_1}_{r_1,\, \mathcal{M}_1}$ from $H^{r_1}(\Lambda_1) \cap H^{\nu_1}_0(\Lambda_1)$ to $P^{\nu_1}_{\mathcal{M}_1}(\Lambda_1)$ such that for any $u\in H^{r_1}(\Lambda_1) \cap H^{\nu_1}_0(\Lambda_1),$ we have $\Vert u-\Pi^{\nu_1}_{r_1,\, \mathcal{M}_1} u \Vert_{{^c}H^{\nu_1}(\Lambda_1)}\leq c_1 \mathcal{M}_1^{\nu_1-r_1} \Vert u \Vert_{H^{r_1}(\Lambda_1)}$, where $c_1$ is a positive constant.
\end{thm}
\begin{thm}
	\label{err_2}
	\cite{kharazmi2016petrov} Let $r_0 \geq 1 $, $r_0\neq \mathcal{N}+\frac{1}{2}$. There exists an operator $\Pi^{\varphi}_{r_0,\, \mathcal{N}}$ from $H^{r_0}(I) \cap {^{l,\mathfrak{D}}}H^{\varphi}(I)$ to $P^{\varphi}_{\mathcal{N}}(\Lambda_1)$ such that for any $u\in H^{r_0}(I) \cap {^{l,\mathfrak{D}}}H^{\varphi}(I)$, we have $$\Vert u-\Pi^{\varphi}_{r_0,\, \mathcal{N}} u \Vert_{{^l}H^{\varphi}(I)}^2\leq c_0 \, \mathcal{N}^{-2r_0} \int_{\tau^{min}}^{\tau^{max}} \varphi(\tau) \,\mathcal{N}^{2\tau}\Vert u \Vert_{H^{r_0}(I)}d\tau,$$ where $c_0$ is a positive constant and $0<\varphi(\tau)\in L^1((\tau^{min},\tau^{max}))$.
\end{thm}
In the following, employing Theorems \ref{err_1} and \ref{err_2} and also Theorem $5.3$ from \cite{samiee2017unified1}, we study the properties of higher-dimensional approximation operators in the following Lemmas.
\begin{thm}
	\label{err_3}
	Let $r_1 \geq 1 $, $r_1\neq \mathcal{M}_1+\frac{1}{2}$. There exists a projection operator $\Pi^{\rho_1}_{r_1,\, \mathcal{M}_1}(I_1)$ from $H^{r_1}(I_1) \cap {^{l,\mathfrak{D}}}H^{\rho_1}(I_1)$ to $P^{\rho_1}_{\mathcal{M}_1}(I_1)$ such that for any $u\in H^{r_1}(I_1) \cap {^{l,\mathfrak{D}}}H^{\rho_1}(I_1)$, we have $$\Vert u-\Pi^{\rho_1}_{r_1,\, \mathcal{M}_1} u \Vert_{{^{l,\mathfrak{D}}}H^{\rho_1}(I_1)}^2\leq   \mathcal{M}_1^{-2r_1} \int_{\nu_1^{min}}^{\nu_1^{max}} \rho_1(\nu_1) \,\mathcal{M}_1^{2\nu_1}\Vert u \Vert_{H^{r_1}(I_1)}d\nu_1,$$ where $0<\rho_1(\nu_1)\in L^1((\nu_1^{min},\nu_1^{max}))$.
\end{thm}
\begin{proof}
	From Theorem \ref{err_1} for $u\in H^{r_1}\cap\, {^c}H^{\nu_1}$ we have $\Vert u-\Pi^{\nu_1}_{r_1,\, \mathcal{M}_1} u \Vert_{{^c}H^{\nu_1}(\Lambda_1)}\leq \mathcal{M}_1^{\nu_1-r_1} \Vert u \Vert_{H^{r_1}(\Lambda_1)}.$ Therefore, for $u\in H^{r_1}(I_1) \cap {^l}H^{\rho_1}(I_1)$ we have
	\begin{align}
	& \Vert u-\Pi^{\rho_1}_{r_1,\, \mathcal{M}_1} u \Vert_{{^{l,\mathfrak{D}}}H^{\rho_1}(I_1)}^2 = \int_{\nu_1^{min}}^{\nu_1^{max}} \rho_1(\nu_1) \,\Vert u-\Pi^{\nu_1}_{r_1,\, \mathcal{M}_1} u \Vert_{{^c}H^{\nu_1}(\Lambda_1)}^2 d\nu_1
	\nonumber
	\\
	\leq & \int_{\nu_1^{min}}^{\nu_1^{max}} \rho_1(\nu_1) \, \mathcal{M}_1^{2\nu_1-2r_1} \Vert u \Vert_{H^{r_1}(\Lambda_1)}^2 d\nu_1 
	=
	\mathcal{M}_1^{-2r_1} \int_{\nu_1^{min}}^{\nu_1^{max}} \rho_1(\nu_1) \,\mathcal{M}_1^{2\nu_1}\Vert u \Vert_{H^{r_1}(I_1)}d\nu_1.
	\nonumber
	\end{align}
\end{proof}

\begin{lem}
	\label{errr_3}
	Let the real-valued $1\leq r_1, \, r_2$ and $\Omega=I_1 \times I_2$. If $u \in {^{l,\mathfrak{D}}}H^{\rho_2}(I_2,H^{r_1}(I_1))\cap H^{r_2}(I_2,{^{l,\mathfrak{D}}}H^{\rho_1}_0(I_1))$, then
	\begin{eqnarray}
	\label{err_3_1}
	&&\Vert u- \Pi^{\rho_1}_{r_1,\, \mathcal{M}_1} \Pi^{\rho_2}_{r_2,\, \mathcal{M}_2} u \Vert_{ \mathcal{B}^{\rho_1,\rho_2}(\Omega)}^2 \leq 
	\nonumber
	\\
	&& \mathcal{M}^{-2r_2}_{2} \int_{\nu_2^{min}}^{\nu_2^{max}} \rho_2(\nu_2) \,\Big(  \mathcal{M}^{2\nu_2}_{2}\Vert u \Vert_{{}H^{r_2}(I_2,L^2(I_1))} +  \mathcal{M}^{2\nu_2}_{2} \mathcal{M}^{-2r_1}_{1}  \Vert u \Vert_{{}H^{r_2}(I_2,{}H^{r_1}(I_1))}\Big) d\nu_2 
	\nonumber
	\\
	&&
	+  \mathcal{M}^{-2r_1}_{1} \int_{\nu_1^{min}}^{\nu_1^{max}} \rho_1(\nu_1) \,\Big(  \mathcal{M}^{2\nu_1}_{1}\Vert u \Vert_{{}H^{r_1}(I_1,L^2(I_2))} +  \mathcal{M}^{2\nu_1}_{1} \mathcal{M}^{-2r_2}_{2}  \Vert u \Vert_{{}H^{r_1}(I_1,{}H^{r_2}(I_2))}\Big) d\nu_1 
	\nonumber
	\\
	&&
	+ \mathcal{M}^{-2r_2}_{2}  \Vert u \Vert_{{^{\mathfrak{D}}}H^{\rho_1}(I_1,H^{r_2}(I_2))} + \mathcal{M}^{-2r_1}_{1}  \Vert u \Vert_{{^{\mathfrak{D}}}H^{\rho_2}(I_2,H^{r_1}(I_1))}, \quad \quad
	\end{eqnarray}
	where $\Vert \cdot \Vert_{ \mathcal{B}^{\rho_1,\rho_2}(\Omega)}=\big{ \{ } \Vert \cdot \Vert_{{}H^{\rho_1}(I_1,L^2(I_2))}^2+\Vert \cdot \Vert_{{}H^{\rho_2}(I_1,L^2(I_1))}^2\big{ \} }^{\frac{1}{2}}$, $0<\rho_1(\nu_1)\in L^1([\nu_1^{min},\nu_1^{max}])$, and $0<\rho_2(\nu_2)\in L^1([\nu_2^{min},\nu_2^{max}])$.
\end{lem}
\begin{proof}
	For $u \in {^{l,\mathfrak{D}}}H^{\rho_2}(I_2,H^{r_1}(I_1))\cap H^{r_2}(I_2,{}H^{\rho_1}(I_1))$, evidently $u \in {}H^{r_2}(I_2,H^{r_1}(I_1))$, $u \in {}H^{r_2}(I_2,L^2(I_1))$, and $u \in {}H^{r_1}(I_1,L^2(I_2))$.
	Besides, from the definition of $\Vert \cdot \Vert_{ \mathcal{B}^{\rho_1,\rho_2}(\Omega)}$ we have
	\begin{align*}
	& \Vert u- \Pi^{\rho_1}_{r_1,\, \mathcal{M}_1} \Pi^{\rho_2}_{r_2,\, \mathcal{M}_2} u \Vert_{ \mathcal{B}^{\rho_1,\rho_2}(\Omega)}
	\\
	=
	& \big{ \{ } \Vert u- \Pi^{\rho_1}_{r_1,\, \mathcal{M}_1} \Pi^{\rho_2}_{r_2,\, \mathcal{M}_2} u \Vert_{{^{\mathfrak{D}}}H^{\rho_1}(I_1,L^2(I_2))}^2+\Vert u- \Pi^{\rho_1}_{r_1,\, \mathcal{M}_1} \Pi^{\rho_2}_{r_2,\, \mathcal{M}_2} u \Vert_{{^{\mathfrak{D}}}H^{\rho_2}(I_1,L^2(I_1))}^2\big{ \} }^{\frac{1}{2}}.
	\end{align*}
	Following Lemma 5.3 in \cite{samiee2017unified1} and Theorem \ref{err_3}, $\Vert u- \Pi^{\rho_1}_{r_1,\, \mathcal{M}_1} \Pi^{\rho_2}_{r_2,\, \mathcal{M}_2} u \Vert_{{^{\mathfrak{D}}}H^{\rho_2}(I_2,L^2(I_1))}^2$ can be simplified to
	\begin{align}
	\label{111334}
	&\Vert u- \Pi^{\rho_1}_{r_1,\, \mathcal{M}_1} \Pi^{\rho_2}_{r_2,\, \mathcal{M}_2} u \Vert_{{^{\mathfrak{D}}}H^{\rho_2}(I_2,L^2(I_1))}^2 
	\nonumber
	\\
	= &
	\Vert u- \Pi^{\rho_2}_{r_2,\, \mathcal{M}_2} u + \Pi^{\rho_2}_{r_2,\, \mathcal{M}_2} u- \Pi^{\rho_1}_{r_1,\, \mathcal{M}_1} \Pi^{\rho_2}_{r_2,\, \mathcal{M}_2} u \Vert_{{^{\mathfrak{D}}}H^{\rho_2}(I_2,L^2(I_1))}^2
	\nonumber
	\\
	\leq  &
	\Vert u- \Pi^{\rho_2}_{r_2,\, \mathcal{M}_2} u \Vert_{{^{\mathfrak{D}}}H^{\rho_2}(I_2,L^2(I_1))}^2 + \Vert \Pi^{\rho_2}_{r_2,\, \mathcal{M}_2} u- \Pi^{\rho_1}_{r_1,\, \mathcal{M}_1} \Pi^{\rho_2}_{r_2,\, \mathcal{M}_2} u \Vert_{{^{\mathfrak{D}}}H^{\rho_2}(I_2,L^2(I_1))}^2
	\nonumber
	\\
	\leq &
	\mathcal{M}^{-2r_2}_2 \int_{\nu_2^{min}}^{\nu_2^{max}} \rho_2(\nu_2) \, \mathcal{M}^{2\nu_2}_2 \Vert u \Vert_{{}H^{r_2}(I_2,L^2(I_1))}^2 d\nu_2 
	\nonumber
	\\
	& \,\,
	+ \Vert (\Pi^{\rho_2}_{r_2,\, \mathcal{M}_2} - \mathcal{I}) (u- \Pi^{\rho_1}_{r_1,\, \mathcal{M}_1} u) \Vert_{{^{\mathfrak{D}}}H^{\rho_2}(I_2,L^2(I_1))}^2
	+ 
	\Vert u-\Pi^{\rho_1}_{r_1,\, \mathcal{M}_1} u \Vert_{{^{\mathfrak{D}}}H^{\rho_2}(I_2,L^2(I_1))}^2
	\nonumber
	\\
	\leq &
	\mathcal{M}^{-2r_2}_2 \int_{\nu_2^{min}}^{\nu_2^{max}} \rho_2(\nu_2) \, \mathcal{M}^{2\nu_2}_2 \Vert u \Vert_{{}H^{r_2}(I_2,L^2(I_1))}^2 d\nu_2  
	\nonumber
	\\
	& \,\,
	+  \mathcal{M}^{-2r_2}_{2} \mathcal{M}^{-2r_1}_{1}  \int_{\nu_2^{min}}^{\nu_2^{max}} \rho_2(\nu_2) \, \mathcal{M}^{2\nu_2}_2 \Vert u \Vert_{{}H^{r_2}(I_2,H^{r_1}(I_1))}^2 d\nu_2   
	+   
	\mathcal{M}^{-2r_1}_{1}  \Vert u \Vert_{{^{\mathfrak{D}}}H^{\rho_2}(I_2,{}H^{r_1}(I_1))}^2,
	\end{align}
	where $\mathcal{I}$ is the identity operator. Furthermore,
	\begin{align}
	\label{11133}
	&\Vert u- \Pi^{\rho_1}_{r_1,\, \mathcal{M}_1} \Pi^{\rho_2}_{r_2,\, \mathcal{M}_2} u \Vert_{L^2(I_2,{}H^{\rho_1}(I_1))}^2  
	\nonumber
	\\
	= & \Vert u- \Pi^{\rho_1}_{r_1,\, \mathcal{M}_1} u + \Pi^{\rho_1}_{r_1,\, \mathcal{M}_1} u- \Pi^{\rho_1}_{r_1,\, \mathcal{M}_1} \Pi^{\rho_2}_{r_2,\, \mathcal{M}_2} u \Vert_{{^{\mathfrak{D}}}H^{\rho_1}(I_1,L^2(I_2))}^2
	\nonumber
	\\
	\leq  &
	\Vert u- \Pi^{\rho_1}_{r_1,\, \mathcal{M}_1} u \Vert_{{^{\mathfrak{D}}}H^{\rho_1}(I_1,L^2(I_2))}^2 + \Vert \Pi^{\rho_1}_{r_1,\, \mathcal{M}_1} u- \Pi^{\rho_1}_{r_1,\, \mathcal{M}_1} \Pi^{\rho_2}_{r_2,\, \mathcal{M}_2} u \Vert_{{^{\mathfrak{D}}}H^{\rho_1}(I_1,L^2(I_2))}^2
	\nonumber
	\\
	\leq &
	\mathcal{M}^{-2r_1}_1 \int_{\nu_1^{min}}^{\nu_1^{max}} \rho_1(\nu_1) \, \mathcal{M}^{2\nu_1}_1 \Vert u \Vert_{{}H^{r_1}(I_1,L^2(I_2))}^2 d\nu_1 
	\nonumber
	\\
	& \,\,
	+ \Vert (\Pi^{\rho_1}_{r_1,\, \mathcal{M}_1} - \mathcal{I}) (u- \Pi^{\rho_2}_{r_2,\, \mathcal{M}_2} u) \Vert_{{^{\mathfrak{D}}}H^{\rho_1}(I_1,L^2(I_2))}^2
	+ \Vert u-\Pi^{\rho_2}_{r_2,\, \mathcal{M}_2} u \Vert_{{^{\mathfrak{D}}}H^{\rho_1}(I_1,L^2(I_2))}^2
	\nonumber
	\\
	\leq &
	\mathcal{M}^{-2r_1}_1 \int_{\nu_1^{min}}^{\nu_1^{max}} \rho_1(\nu_1) \, \mathcal{M}^{2\nu_1}_1 \Vert u \Vert_{{}H^{r_1}(I_1,L^2(I_2))}^2 d\nu_1   
	\nonumber
	\\
	& \,\,
	+  \mathcal{M}^{-2r_2}_{2} \mathcal{M}^{-2r_1}_{1}  \int_{\nu_1^{min}}^{\nu_1^{max}} \rho_1(\nu_1) \, \mathcal{M}^{2\nu_1}_1 \Vert u \Vert_{{}H^{r_1}(I_1,H^{r_2}(I_2))}^2 d\nu_1   
	+  \mathcal{M}^{-2r_2}_{2}  \Vert u \Vert_{{^{\mathfrak{D}}}H^{\rho_1}(I_1,{}H^{r_2}(I_2))}^2.
	\end{align}
	Therefore, \eqref{err_3_1} can be derived immediately from \eqref{11133} and \eqref{111334}.
\end{proof}

\noindent Likewise, Lemma \ref{err_3} can be easily extended to the $d$-dimensional approximation operator as
\begin{align}
\label{err_5}
&\Vert u-\Pi_d^h u \Vert_{{^{\mathfrak{D}}}H^{\rho_i}(I_i,L^2(\Lambda_d^i))}^2 
\nonumber
\\
&\leq 
\mathcal{M}_i^{-2r_i} \int_{\nu_i^{min}}^{\nu_i^{max}} \rho_i(\nu_i) \, \mathcal{M}^{2\nu_i}_i \Vert u \Vert_{{}H^{r_i}(I_i,L^2(\Lambda_d^i))}^2 d\nu_i+\sum_{\underset{j\neq i}{j=1}}^d \mathcal{M}_j^{-2r_j}\Vert u \Vert_{{^{\mathfrak{D}}}H^{\rho_i}(I_i,H^{r_j}(I_j,L^2(\Lambda_d^{i,j})))}^2
\nonumber
\\
&
+\mathcal{M}_i^{-2r_i} \int_{\nu_i^{min}}^{\nu_i^{max}} \rho_i(\nu_i) \, \mathcal{M}^{2\nu_i}_i \sum_{\underset{j\neq i}{j=1}}^d \mathcal{M}_j^{-2r_j}\Vert u \Vert_{{}H^{r_i}(I_i,H^{r_j}(I_j,L^2(\Lambda_d^{i,j})))}^2 d\nu_i
\nonumber
\\
&+\sum_{\underset{k\neq i}{k=1}}^d \sum_{\underset{j\neq i, \, k}{j=1}}^d \mathcal{M}_j^{-2r_j} \mathcal{M}_k^{-2r_k}\Vert u \Vert_{{^{\mathfrak{D}}}H^{\rho_i}(I_i,H^{r_k,r_j}(I_k\times I_j,L^2(\Lambda_d^{i,j,k}))))}^2
\nonumber
\\
&
+\cdots + \mathcal{M}_i^{-2r_i} \int_{\nu_i^{min}}^{\nu_i^{max}} \rho_i(\nu_i) \, \mathcal{M}^{2\nu_i}_i \prod_{\underset{j\neq i}{j=1}}^{d} \mathcal{M}_j^{-r_j} \Vert u \Vert_{{^c}H^{\nu_i}(I_i,H^{r_1,\cdots,r_d}(\Lambda_d^i)))}^2 d\nu_i,
\end{align}
where $\Pi^{h}_{d}=\Pi^{\rho_1}_{r_1,\, \mathcal{M}_1}\cdots \Pi^{\rho_d}_{r_d,\, \mathcal{M}_d}$.

\begin{thm}
	\label{thmerr}
	Let $1\leq r_i$, $I_0=(0,T)$, $I_i=(a_i,b_i)$, $\Omega=I_0 \times \Big(\prod_{i=1}^{d}I_i\Big)$, $\Lambda_k=\prod_{i=1}^{k}I_i$, $\Lambda_k^j=\prod_{\underset{i\neq j}{i=1}}^{k}I_i$ and $\frac{1}{2}<\nu_i^{min}<\nu_i^{max}<1$ for $i=1,\cdots,d$. If 
	\begin{align*}
	u \in  \Big(\overset{d}{\underset{i=1}{\cap}} H^{r_0}(I_0,{^{\mathfrak{D}}}H^{\rho_i}(I_i,{}H^{r_1,\cdots,r_{i-1},r_{i+1},\cdots,r_d}(\Lambda_d^i))\Big)\cap {^{l,\mathfrak{D}}}H^{\varphi}(I_0,H^{r_1,\cdots,r_d}(\Lambda_d)),
	\end{align*}
	then,
	\begin{align}
	\label{them111}
	\Vert u- \Pi^{\varphi}_{r_0,\, \mathcal{N}} \Pi^{h}_{d} u \Vert_{ \mathcal{B}^{\tau,\nu_1,\cdots,\nu_d}(\Omega)}^2
	\nonumber 
	& \leq
	\mathcal{N}^{-2r_0} \int_{\tau^{min}}^{\tau^{max}} \varphi(\tau) \,\mathcal{N}^{2\tau} \Vert u \Vert_{{}H^{r_0}(I_0,L^2(\Lambda_d))} d\tau
	\nonumber \\
	& 
	+\mathcal{N}^{-2r_0} \int_{\tau^{min}}^{\tau^{max}} \varphi(\tau) \,\mathcal{N}^{2\tau}\sum_{j=1}^{d} \mathcal{M}_j^{-2r_j} \Vert u \Vert_{{}H^{r_0}(I_0,H^{r_j}(I_j,L^2(\Lambda_d^j)))}^2 d\tau+                       \cdots 
	\nonumber
	\\
	& 
	+ \mathcal{N}^{-2r_0} \int_{\tau^{min}}^{\tau^{max}} \varphi(\tau) \,\mathcal{N}^{2\tau} \Big( \prod_{\underset{}{j=1}}^{d} \mathcal{M}_j^{-2r_j} \Big) \Vert u \Vert_{{}H^{r_0}(I_0,H^{r_1,\cdots,r_d}(\Lambda_d)))}d\tau
	\nonumber
	\\
	&
	+
	\sum_{i=1}^{d} \int_{\nu_i^{min}}^{\nu_i^{max}} \rho_i(\nu_i) \Big{\{}\mathcal{M}_i^{2\nu_i-2r_i} \Vert u \Vert_{{}H^{r_i}(I_i,L^2(\Lambda_d^i\times I_0))}+\cdots 
	\nonumber
	\\
	& 
	+ \mathcal{M}_i^{2\nu_i-2r_i} \Big( \prod_{\underset{j\neq i, \, k}{j=1}}^{d} \mathcal{M}_j^{-2r_j} \Big) \Vert u \Vert_{{}H^{r_i}(I_i,H^{r_1,\cdots,r_d}(\Lambda_d^i,L^2(I_0)))}\Big{\}}d\nu_i,
	\end{align}
	where $\Pi^{h}_{d}=\Pi^{\rho_1}_{r_1,\, \mathcal{M}_1}\cdots \Pi^{\rho_d}_{r_d,\, \mathcal{M}_d}$ and $\beta$ is a real positive constant.
\end{thm}
\begin{proof}
	Directly from \eqref{norm_22213} we conclude that
	\begin{eqnarray}
	\label{2233}
	&&\Vert u
	%- \Pi^{\tau}_{r_0,\, \mathcal{N}} \Pi^{h}_{d} u 
	\Vert_{ \mathcal{B}^{\tau,\nu_1,\cdots,\nu_d}(\Omega)}^2 
	\leq \Vert u \Vert_{{^l}H^{\tau}(I_0,L^2(\Lambda_d))}^2+\sum_{i=1}^{d}\Vert u \Vert_{L^2(I_0,{^{\mathfrak{D}}}H^{\rho_i}(I_i,L^2(\Lambda_d^i)))}^2.
	\nonumber
	\end{eqnarray}
	Next, it follows from Theorem \ref{err_2} that
	\begin{align}
	\label{err_6}
	& \Vert u-\Pi^{\varphi}_{r_0,\, \mathcal{N}}\Pi_d^h u \Vert_{{^{l,\mathfrak{D}}}H^{\varphi}(I_0,L^2(\Lambda_d))}^2 
	\nonumber 
	\\
	& 
	\leq \mathcal{N}^{-2r_0} \int_{\tau_{min}}^{\tau_{max}} \varphi(\tau)\, \mathcal{N}^{2\tau}\,\Bigg{[} \Vert u \Vert_{{}H^{r_0}(I_0,L^2(\Lambda_d))}^2 +
	\sum_{j=1}^{d}\mathcal{M}_j^{-r_j} \Vert u \Vert_{{}H^{r_0}(I_0,H^{r_j}(I_j,L^2(\Lambda_d)))}^2 + \cdots 
	\nonumber
	\\
	&
	\qquad \qquad \qquad \qquad \qquad \qquad+
	\Big( \prod_{\underset{}{j=1}}^{d} \mathcal{M}_j^{-r_j} \Big) \Vert u \Vert_{{}H^{r_0}(I_0,H^{r_1,\cdots,r_d}(\Lambda_d)))}\Bigg{]}
	\, d\tau.
	\end{align}
	Therefore, \eqref{them111} is obtained immediately from \eqref{err_5} and \eqref{err_6}.	
\end{proof}

\begin{rem}
	Since the \text{inf-sup} condition holds (see Theorem \ref{Thm: inf-sup_3}), by Lemma \ref{continuity_lem},
	%the Banach-Ne\v{c}as-Babu\v{s}ka theorem \citep{ern2013theory},
	the error in the numerical scheme is less than or equal to a constant times the projection error. Hence the results above imply the spectral accuracy of the scheme.
\end{rem}

%
%%%%%%%%%%%%%%%%%%%%%%%%%%%%%%%%%
\section{Numerical Tests}
\label{Illustration of Results}
%%%%%%%%%%%%%%%%%%%%%%%%%%%%%%%%%
%
We provide several numerical examples to investigate the performance of the proposed scheme. We consider a $(1+d)$-dimensional fully distributed diffusion problem with left-sided derivative by letting $c_{l_i}=c_{r_i}=\kappa_{r_i}=0$, $\kappa_{l_i}=1$, $0<2\tau^{min}<2\tau^{max}<1$ and $1<2\nu^{min}_i<2\nu^{max}_i<2$ in \eqref{Eq: infinit-dim PG method_1111} for $i=1,\cdots,d$, where the computational domain is $\Omega = (0,2) \times \prod_{i=1}^{d} (-1,1)$. We report the measured $L^{\infty}$ error, $\Vert e\Vert_{L^{\infty}}=\Vert u_N - u^{ext}\Vert_{L^{\infty}}$.

In each of the following test cases, we use the method of fabricated solutions to construct the load vector, given an exact solution $u^{ext}$. Here, we assume $u^{ext} = u_t  \times \prod_{i=1}^{d} u_{x_i}$. We project the spatial part in each dimension, $u_{x_i}$, on the spatial bases, and then, construct the load vector by plugging the projected exact solution into the weak form of problem. This helps us take the fractional derivative of exact solution more efficiently, while by truncating the projection with a sufficient number of terms, we make sure that the corresponding projection error does not dominantly propagate into the convergence analysis of numerical scheme.

\vspace{0.2 in}
\noindent \textbf{Case I:} We consider a smooth solution in space with finite regularity in time as 
\begin{equation}
\label{test I}
u^{ext}=t^{p_1 + \alpha}\times \Big( (1+x_1)^{p_2} (1-x_1)^{p_3} \Big)
\end{equation} 
to investigate the spatial/temporal \textit{p}-refinement. We allow the singularity to take order of $\alpha = 10^{-4}$, while $p_1$, $p_2$, and $p_3$ take some integer values. We show the $L^{\infty}$-error for different test cases in Fig. \ref{Fig: Case I}, where by tuning the fractional parameter of the temporal basis, we can accurately capture the singularity of the exact solution, when the approximate solution converges as we increase the expansion order. In each case of spatial/temporal \textit{p}-refinement, we choose sufficient number of bases in other directions to make sure their corresponding error is of machine precision order. We also note that the proposed method efficiently converges, however, as the order of singularity $\alpha$ increases, the rate of convergences slightly drops, see the dashed lines in Fig. \ref{Fig: Case I}.

%In Table \ref{my-label} Case A, we study the spectral convergence rate of the method for smooth solutions with different regularities in \eqref{test I} for (1+1)-dimensional diffusion problem. 
Considering $\alpha=10^{-4}$, $p_1=2$, $p_2=p_3=2$ in \eqref{test I}, and the temporal order of expansion being fixed ($\mathcal{N}=4$) in the spatial \textit{p}-refinement, we get 
%the error in spatial direction approaches to the exact solution sufficiently and hence 
the rate of convergence as a function of the minimum regularity in the spatial direction. From Theorem \ref{thmerr}, the rate of convergence is bounded by the spatial approximation error, i.e. $\Vert e \Vert_{L^2(\Omega)}\leq \Vert e \Vert_{L^\infty(\Omega)} \leq  \mathcal{M}_1^{-2r_1} \int_{\nu_1^{min}}^{\nu_1^{max}} \rho_1(\nu_1) \,\mathcal{M}_1^{2\nu_1} \Vert u \Vert_{{}H^{r_1}(I_1,L^2(I_0))} d\nu_1$, where $r_1=p_2+\frac{1}{2}-\epsilon$ is the minimum regularity of the exact solution in the spatial direction for $\epsilon<\frac{1}{2}$. Conforming to Theorem \ref{thmerr}, the practical rate of convergence $\bar{r}_1=16.05$ in $\Vert e \Vert_{L^\infty(\Omega)}$ is greater than $r_1\approx 2.50$.

%
%
% %which justifies the greater value of $r$ in $\Vert e \Vert_{L^2(\Omega)}$ relative to the corresponding one in $\Vert e \Vert_{\mathcal{B}^{\tau,\nu}(\Omega)}$ in Table \ref{Table: Higher-Dimensional FPDEs}.
%
%Similarly, we compare the practical rate of convergence $\bar{r}_0$ and the estimated one $r_0$, provided in \eqref{thmerr}, for solutions with different regularities in  \eqref{}. We kept the temporal order of expansion fixed ($\mathcal{N}=4$) that the approximated solution $u_N$ in temporal direction approaches to the exact solution sufficiently and hence $\Vert e \Vert_{L^{\infty}(\Omega)}$ is dominated by the spatial error, i.e., $\Vert e \Vert_{L^2(\Omega)}\leq \Vert e \Vert_{L^\infty(\Omega)} \leq  \mathcal{M}_1^{-2r_1} \int_{\nu_1^{min}}^{\nu_1^{max}} \rho_1(\nu_1) \,\mathcal{M}_1^{2\nu_1} \Vert u \Vert_{{}H^{r_1}(I_1,L^2(I_0))} d\nu_1$, where $r_1=p_2+\frac{1}{2}-\epsilon$. Conforming to Theorem \ref{thmerr}, the practical rates of convergence $\bar{r}_1$, provided in Table \ref{Table: Higher-Dimensional FPDEs} Case B, in $\Vert e \Vert_{L^\infty(\Omega)}$ is greater than $r_1 \approx $.
%******************************************************************************************
%
\begin{figure}[h]
	\centering
	\begin{subfigure}{0.32\textwidth}
		\centering
		\includegraphics[width=1\linewidth]{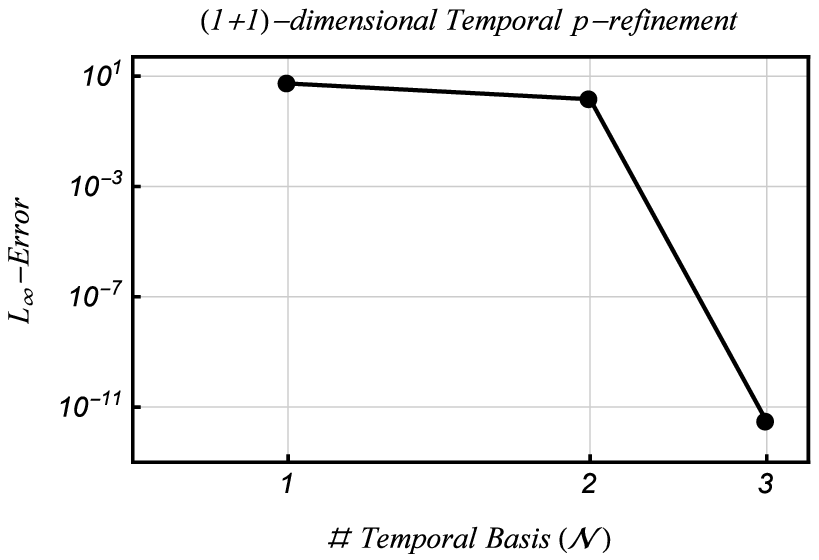}
		%		\caption{}
		%\label{fig:}
	\end{subfigure}
	\begin{subfigure}{0.32\textwidth}
		\centering
		\includegraphics[width=1\linewidth]{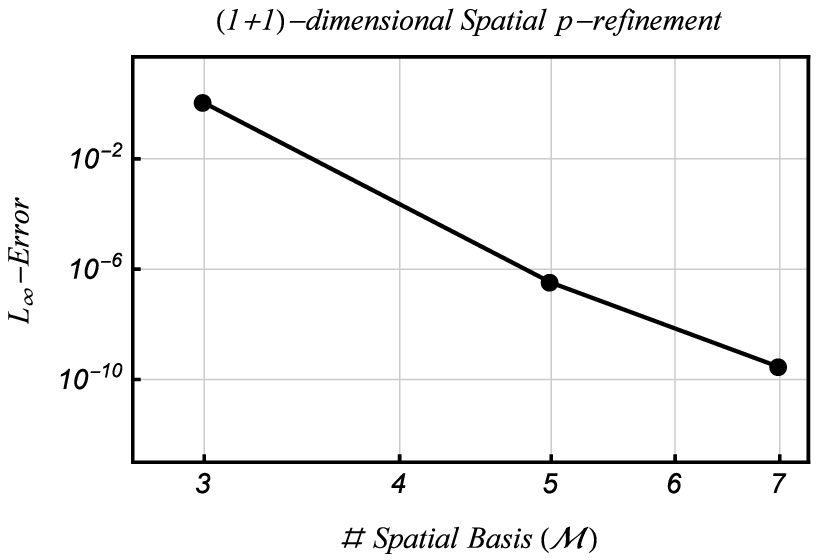}
		%		\caption{}
		%\label{fig:}
	\end{subfigure}
	\begin{subfigure}{0.32\textwidth}
		\centering
		\includegraphics[width=1\linewidth]{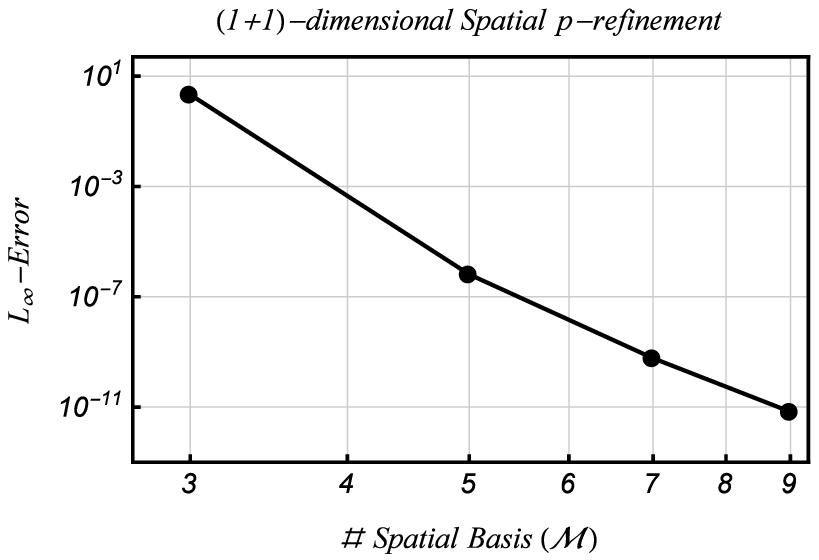}
		%		\caption{}
		%\label{fig:}
	\end{subfigure}
	%	\vspace{-0.15 in}
	\caption{Temporal/Spatial \textit{p-refinement} for case I with singularity of order $\alpha = 10^{-4}$. (Left): $p_1=3$, $p_2=p_3=2$, and expansion order of $\mathcal{N} \times 9$. (Middle): $p_1=2$, $p_2=p_3=2$, and expansion order of $3 \times \mathcal{M}$. (Right): $p_1=3$, $p_2=p_3=2$, and expansion order of $4 \times \mathcal{M}$. }
	\label{Fig: Case I}
\end{figure}
%
%******************************************************************************************
%

%\vspace{0.1 in}
\noindent \textbf{Case II:} We consider $u^{ext}=t^{p_1 + \alpha} \sin ( 2 \pi x_1 )$, where $p_1=3$, and let $\alpha = 0.1$ and $\alpha = 0.9$. We set the number of temporal basis functions, $\mathcal{N} = 4$, and show the convergence of approximate solution by increasing the number of spatial basis, $\mathcal{M}$ in Fig. \ref{Fig: Case II}. The main difficulty in this case is the construction of the load vector. To accurately compute the integrals in the construction of the load vector, we project the spatial part of the forcing function, $\sin ( 2 \pi x_1 )$, on the spatial bases. To make sure that the corresponding error is of machine-precision order and thus, not dominant, we truncate the projection at 25 terms, where there error is of order $10^{-16}$. Therefore, the quadrature rule over derivative order should be performed for 25 terms rather than only a single $\sin ( 2 \pi x_1 )$ term. This will increase the computational cost.
%
%******************************************************************************************
%
\begin{figure}[h]
	\centering
	\begin{subfigure}{0.45\textwidth}
		\centering
		\includegraphics[width=1\linewidth]{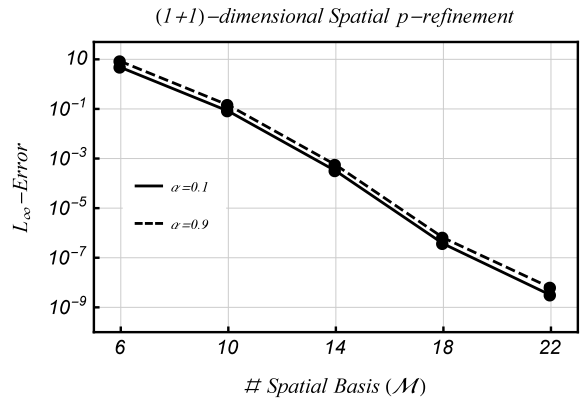}
		%\caption{}
		%\label{fig:}
	\end{subfigure}
	\caption{Spatial \textit{p}-refinement for case II, $p_1=3$, $\alpha = 0.1$, and $\alpha = 0.9$.}
	\label{Fig: Case II}
\end{figure}
%
%******************************************************************************************
%

%\vspace{0.2 in}
\noindent \textbf{Case III:} (High-dimensional  \textit{p}-refinement) We consider the exact solution of the form \begin{equation}
\label{test III}
u^{ext}=t^{p_1 + \alpha}\times \prod_{i=1}^{3} (1+x_i)^{p_{2i}} (1-x_i)^{p_{2i+1}}
\end{equation}
with singularity of order $\alpha = 10^{-4}$, where $p_1=3$, and $p_{2i}=p_{2i+1}=1$. Similar to previous cases, we set the number of temporal bases, $\mathcal{N} = 4$, and study convergence by uniformly increasing the number of spatial bases in all dimensions. Fig. \ref{Fig: Case III} shows the results for $(1+2)$-dimensional and $(1+3)$-dimensional problems with expansion order of $\mathcal{N} \times \mathcal{M}_1 \times \mathcal{M}_2$, and $\mathcal{N} \times \mathcal{M}_1 \times \mathcal{M}_2 \times \mathcal{M}_3$, respectively. Following Case I, the computed rate of convergence $\bar{r}_1=\bar{r}_2=\bar{r}_3=16.13$ in \eqref{test III} for $\alpha=10^{-4}$ is greater than the minimum regularity of the exact solution $r \approx 2.05$, which is in agreement with Theorem \ref{thmerr}.
%
%******************************************************************************************
%
\begin{figure}[h]
	\centering
	\begin{subfigure}{0.4\textwidth}
		\centering
		\includegraphics[width=1\linewidth]{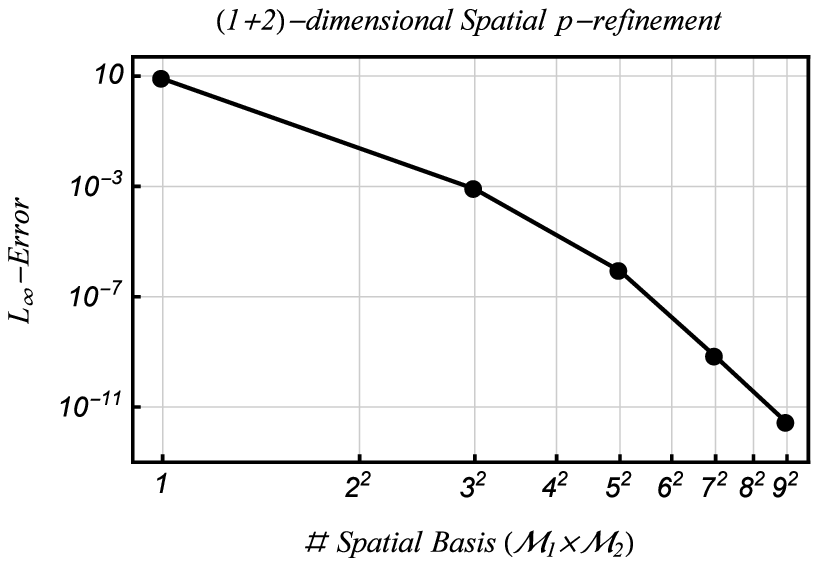}
		%		\caption{}
		%\label{fig:}
	\end{subfigure}
	\begin{subfigure}{0.4\textwidth}
		\centering
		\includegraphics[width=1\linewidth]{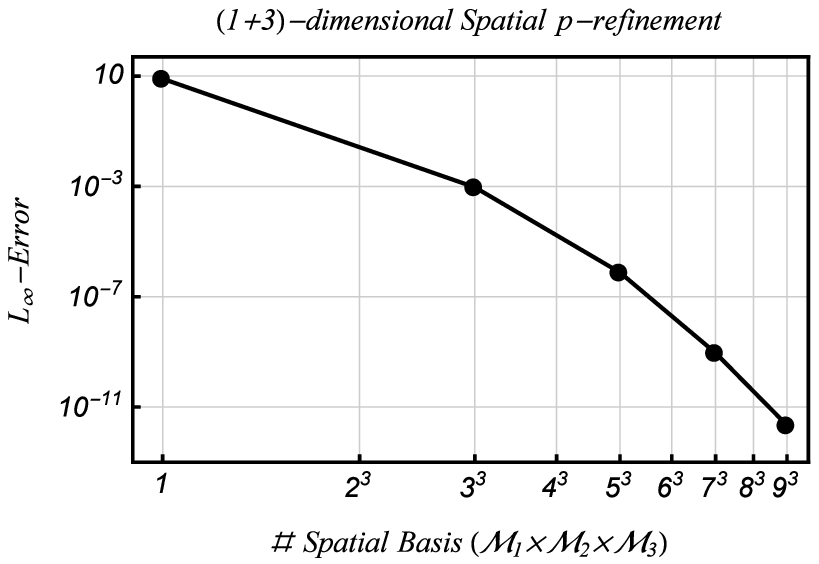}
		%		\caption{}
		%\label{fig:}
	\end{subfigure}
	\caption{Spatial \textit{p}-refinement for case III with singularity of order $\alpha = 10^{-4}$. (Left): $(1+2)$-dimensional, $p_1=3$, $p_{2i}=p_{2i+1}=1$, where the expansion order is $\mathcal{N} \times \mathcal{M}_1 \times \mathcal{M}_2$. (Left): $(1+3)$-dimensional, $p_1=3$, $p_{2i}=p_{2i+1}=1$, where the expansion order is $\mathcal{N} \times \mathcal{M}_1 \times \mathcal{M}_2 \times \mathcal{M}_3$. }
	\label{Fig: Case III}
\end{figure}
%
%******************************************************************************************
%

In addition to the convergence study, we examine the efficiency of the developed method and fast solver by comparing the CPU times for $(1+1)$-, $(1+2)$-, and $(1+3)$-dimensional space-time hypercube domains in case III. The computed CPU times are obtained on an INTEL(XEON E52670) processor of 2.5 GHZ, and reported in Table \ref{Table: Cpu Time}.
%%
%\begin{table}[h]
%	\centering
%	\captionsetup{justification=centering}
%	\caption{\label{Table: Cpu Time} CPU time, PG spectral method for fully distributed (1+d)-dimensional diffusion problems. $u^{ext}=t^{p_1 + \alpha}\times \prod_{i=1}^{3} (1+x_i)^{p_{2i}} (1-x_i)^{p_{2i+1}}$, where $\alpha = 10^{-4}$, $p_1 = 3$, and the expansion order is $4 \times 11 ^ d$.}
%%
%\scalebox{0.7}{
%\begin{tabular}
%	%
%	{cccc}
%	%
%	\multicolumn{4}{c}{$p_{2i}=p_{2i+1}=2$} \\ 	\hline \hline
%	%
%	\hspace*{1.0in}& d=1 & d=2 & d=3    \\ 	\hline
%	%
%	CPU Time $[$Sec$]$ & $1546.81$ & $1735.03$ & $2358.67$ \\ \hline
%	%
%	$\Vert e \Vert_{L^{\infty}(\Omega)}$ & $6.84 \times 10^{-12}$  & $4.45 \times 10^{-12}$ & $3.27 \times 10^{-12}$ \\ \hline
%	%
%\end{tabular}
%}
%%
%\scalebox{0.7}{
%\begin{tabular}
%	%
%	{cccc}
%	%
%%	\vspace{0.1 in}\\
%	\multicolumn{4}{c}{$p_{2i}=p_{2i+1}=3$} \\ 	\hline \hline 
%	%
%	\hspace*{1.0in}& d=1 & d=2 & d=3 \\	\hline
%	%
%	CPU Time $[$Sec$]$ & $1596.16$ & $1786.61$ & $2407.22$ \\ \hline
%	%
%	$\Vert e \Vert_{L^{\infty}(\Omega)}$ & $6.27 \times 10^{-12}$  & $3.86 \times 10^{-12}$ & $2.71 \times 10^{-12}$  \\ \hline
%	%
%\end{tabular}
%}
%%
%\end{table}

%
%%%%%%%%%%%%%%%%%%%%%%%%%%%%%%%%%%%%%%
%
\begin{table}[h]
	\centering
	\captionsetup{justification=centering}
	\caption{\label{Table: Cpu Time} CPU time, PG spectral method for fully distributed (1+d)-dimensional diffusion problems. $u^{ext}=t^{p_1 + \alpha}\times \prod_{i=1}^{3} (1+x_i)^{p_{2i}} (1-x_i)^{p_{2i+1}}$, where $\alpha = 10^{-4}$, $p_1 = 3$, and the expansion order is $4 \times 11 ^ d$.}
	\scalebox{0.7}{
		\begin{tabular}
			{l  c|c|c || c|c|c|}
			& \multicolumn{3}{c}{$p_{2i}=p_{2i+1}=2$} & \multicolumn{3}{c}{$p_{2i}=p_{2i+1}=3$}\\ 	\cline{2-7} 
			\multicolumn{1}{c|}{} & d=1 & d=2 & d=3 & d=1 & d=2 & d=3   \\  	\hline 
			\multicolumn{1}{|c||}{CPU Time $[$Sec$]$} & $1546.81$ & $1735.03$ & $2358.67$ & $1596.16$ & $1786.61$ & $2407.22$\\  \hline
			\multicolumn{1}{|c||}{$\Vert e \Vert_{L^{\infty}(\Omega)}$} & $6.84 \times 10^{-12}$  & $4.45 \times 10^{-12}$ & $3.27 \times 10^{-12}$ & $6.27 \times 10^{-12}$  & $3.86 \times 10^{-12}$ & $2.71 \times 10^{-12}$ \\ \hline
		\end{tabular}
	}
\end{table}
\section{Summary}
\label{Summary and Discussion}
%%%%%%%%%%%%%%%%%%%%%%%%%%%%%%%%%
%
We developed a unified Petrov-Galerkin spectral method for fully distributed-order PDEs with constant coefficients on a ($1+d$)-dimensional \textit{space-time} hypercube, subject to homogeneous Dirichlet initial/boundary conditions. We obtained the weak formulation of the problem, and proved the well-posedness by defining the proper underlying  \textit{distributed Sobolev} spaces and the associated norms. We then formulated the numerical scheme, exploiting Jacobi \textit{poly-fractonomial}s as temporal basis/test functions, and Legendre polynomials as spatial basis/test functions. In order to improve efficiency of the proposed method in higher-dimensions, we constructed a unified fast linear solver employing certain properties of the stiffness/mass matrices, which significantly reduced the computation time. Moreover, we proved stability of the developed scheme and carried out the error analysis. Finally, via several numerical test cases, we examined the practical performance of proposed method and illustrated the spectral accuracy.

%
%%%%%%%%%%%%%%%%%%%%%%%%%%%%%%%%%%
%\section*{Acknowledgement}
%%%%%%%%%%%%%%%%%%%%%%%%%%%%%%%%%%
%%
%
%
%
%%
%%%%%%%%%%%%%%%%%%%%%%%%%%%%%%%%%%%%%%%%%5
\section*{Appendix: Entries of Spatial Stiffness Matrix}
\label{Apx: entries of spatial stiffness matrix}
%%%%%%%%%%%%%%%%%%%%%%%%%%%%%%%%%%%%%%%%%
%
Here, we provide the computation of entries of the spatial stiffness matrix by performing an affine mapping $\vartheta$ from the standard domain $\mu^{stn}_j \in [-1,1]$ to $\mu_j \in [\mu_j^{max},\mu_j^{min}]$.
% the dis integral over the distribution in \eqref{Eq: infinit-dim PG method_1122} can be evaluated by employing a Gauss–Legendre quadrature rule.
%
%\subsection*{$\bullet$ \textbf{Proof of Lemma \ref{lemma31}}}
\begin{lem}
	\label{Thm: Spatial Stiffness Matrix}
	The total spatial stiffness matrix $S^{Tot}_{j}$ is symmetric and its entries can be exactly computed as:
	%
	%To figure out the total spatial stiffness matrix, $S_{\mu_1}$ and $S_{\nu_1}$ are evaluated here. As we know from sec 3,
	\begin{eqnarray}
	\label{Eq: total Lyapunov-2d}
	S^{\, {Tot}}_{j}=c_{l_j} \times S_{l}^{\varrho_j} + c_{r_j} \times S_{r}^{\varrho_j}-\kappa_{l_j} \times S_{l}^{\rho_j}-\kappa_{r_j} \times S_{r}^{\rho_j},
	\end{eqnarray}
	where $j=1,2,\cdots,d$.
	%
	%In this case, the spatial stiffness matrix $S_1^{Tot}$ is symmetric as well as $S_{\mu_1}$ and $S_{\nu_1}$. Moreover, its entries can be computed exactly by employing a Gauss-Jacobi (GLJ) rule with respect to the weight function $(1-\xi)^{-\mu}(1+\eta)^{-\mu}$, where $\xi \in [-1,1]$.
	%
\end{lem}
\begin{proof}
	Regarding the definition of stiffness matrix, we have
	\begin{align}
	\{ S_{l}^{\varrho_j} \}_{r,n}
	&= 
	\int_{-1}^{1} \int_{\mu_j^{min}}^{\mu_j^{max}} \varrho_j(\mu_j^{})
	\prescript{}{-1}{\mathcal{D}}_{\xi_j}^{\mu_j^{}} \Big(\,\phi^{}_n  ( x_j ) \Big)
	\prescript{}{\xi_j}{\mathcal{D}}_{1}^{\mu_j^{}} \Big(\Phi^{}_r ( x_j ) \Big)
	\,
	d{x_j},
	\nonumber
	\\
	&=
	\beta_1 \int_{-1}^{1} \int_{-1}^{1} \varrho_j\big(\vartheta(\mu_j^{stn})\big) \prescript{}{-1}{\mathcal{D}}_{\xi_j}^{\mu_j^{stn}} \Big(\, P_{n+1}({\xi_j}) - P_{n-1}({\xi_j}) \Big)
	\nonumber
	\\
	& \qquad \qquad \quad \times \prescript{}{{\xi}_j}{\mathcal{D}}_{1}^{\mu_j^{stn}} \Big(\, P_{k+1}({\xi_j}) - P_{k-1}({\xi_j}) \Big)
	\,
	d{\xi_j} , 
	\nonumber
	\\
	&=
	\beta_1 \Big[\,\widetilde{S}^{\, \varrho_j}_{r+1,n+1}-\widetilde{S}_{r+1,n-1}^{\, \varrho_j}-\widetilde{S}_{r-1,n+1}^{\, \varrho_j}+\widetilde{S}_{r-1,n-1}^{\, \varrho_j}  \Big],   \quad \quad
	\end{align}
	where $\beta_1= {\widetilde{\sigma}}_r\,\sigma_n \, \Big(\frac{\mu_j^{max}-\mu_j^{min}}{2}\Big)$ and
	\begin{align}
	\widetilde{S}_{r,n}^{\varrho_j} &= \,\int_{-1}^{1}\int_{-1}^{1}
	\varrho_j\big(\vartheta(\mu_j^{stn})\big) \prescript{}{-1}{\mathcal{D}}_{\xi_j}^{\mu_j^{stn}} \Big(\, P_{n} (\xi_j) \Big)
	\prescript{}{\xi_j}{\mathcal{D}}_{1}^{\mu_j^{stn}} \Big(\, P_{r} (\xi_j) \Big)
	\,
	d{\xi_j} \, d\mu_j^{stn}
	\nonumber 
	\\
	&=
	\int_{-1}^{1}
	\varrho_j\big( \vartheta(\mu_j^{stn})\big)\, \frac{\Gamma(r+1)}{\Gamma(r-\mu_j^{stn}+1)}\,  \frac{\Gamma(n+1)}{\Gamma(n-\mu_j^{stn}+1)} \, 
	\nonumber \\ 
	& \, \times \int_{-1}^{1} {(1-\xi_j^2)}^{-\mu_j^{stn}} \, P^{ -\mu_j^{stn},\mu_j^{stn}}_{r} \, P^{ \mu_j^{stn},-\mu_j^{stn}}_{n} d{\xi_j}\, d\mu_j^{stn}.
	\nonumber
	\end{align}
	$\widetilde{S}_{r,n}^{\varrho_j}$ can be computed accurately using Gauss-Legendre (GL) quadrature rules as
	\begin{eqnarray}
	\nonumber
	\widetilde{S}_{r,n}^{\varrho_j^{stn}}
	&=&
	\sum_{q=1}^{Q} \frac{\Gamma(r+1)}{\Gamma(r-\mu_j^{stn}|_{q}+1)} \frac{\Gamma(n+1)}{\Gamma(n-\mu_j^{stn}|_{q}+1)} \varrho_j|_{q}\,
	w_q \,\times
	\nonumber
	\\
	&& \int_{-1}^{1} {(1-\xi_j^2)}^{-\mu_j^{stn}|_{q}} \, P^{ -\mu_j^{stn}|_{q},\mu_j^{stn}|_{q}}_{r}(\xi_j) \, P^{ \mu_j^{stn}|_{q},-\mu_j^{stn}|_{q}}_{n}(\xi_j) d{\xi_j}, \quad
	\end{eqnarray}
	in which $\mathcal{Q} \ge \mathcal{M}_j +2 $ represents the minimum number of GL quadrature points $\{\mu_j^{stn}|_{q}\}_{q=1}^{\mathcal{Q}}$ for \textit{exact} quadrature, and $\{w_q\}_{q=1}^{Q}$ are the corresponding quadrature weights. 
	Exploiting the property of the Jacobi polynomials where $P^{\alpha, \beta}_n(-\xi_j) = (-1)^n P^{ \beta,\alpha}_n(\xi_j)$, we have $\widetilde{S}_{r,n}^{\, \varrho_j^{stn}}=(-1)^{(r+n)}\,\widetilde{S}_{n,r}^{\, \varrho_j^{stn}}$. Following \cite{samiee2017Unified}, ${\widetilde{\sigma}}_r$ and $\sigma_n$ are chosen such that $(-1)^{(n+r)}$ is canceled. Accordingly, $\{ S_{l}^{\varrho_j}\}_{n,r}=\{S_{l}^{\varrho_j}\}_{r,n}=\{ S_{r}^{\varrho_j}\}_{r,n}=\{S_{r}^{\varrho_j}\}_{r,n}$ due to the symmetry of $ S_{l}^{\varrho_j}$ and $S_{r}^{\varrho_j}$. Similarly, we get $\{S_{l}^{\rho_j}\}_{n,r}=\{ S_{l}^{\rho_j}\}_{r,n}=\{ S_{r}^{\rho_j}\}_{n,r}=\{ S_{r}^{\rho_j}\}_{r,n} $. Eventually, we conclude that the stiffness matrix $S^{\, \varrho_j}_{l}$, $S^{\, \varrho_j}_{r}$, $S^{\, \rho_j}_{l}$, $S^{\, \rho_j}_{r}$, and thereby $\{S^{Tot}_{j}\}_{n,r}$ as the sum of symmetric matrices are symmetric.
\end{proof}

%

%%%%%%%%%%%%%%%%%%%%%%%%%%%%%%%%%%%%%%%%%%%%%%%%%%%%%%%%%%%%%%%%%%%%%%%%%%%%%%%%%%%%%%%%%%%%%%%%%%%%%%%%%%%%%%%%%%%%%%%%%%%%%%%%%%%%%%%%
%%%%%%%%%%%%%%%%%%%%%%%%%%%%%%%%%%%%%%%%%%%%%%%%%%%%%%%%%%%%%%%%%%%%%%%%%%%%%%%%%%%%%%%%%%%%%%%%%%%%%%%%%%%%%%%%%%%%%%%%%%%%%%%%%%%%%%%%

\newpage
\bibliographystyle{siam}
\bibliography{RFSLP_Refs2}

\end{document}